\documentclass[psamsfonts]{amsart}

\pdfoutput=1

\usepackage[dvipsnames]{xcolor}

\usepackage{amsmath}
\usepackage{amssymb}
\usepackage{amsthm}
\usepackage{amsrefs}
\usepackage{amsfonts}
\usepackage{mathrsfs}
\usepackage{mathtools}
\usepackage{stmaryrd}
\usepackage[mathcal]{eucal}
\usepackage[all]{xy}
\usepackage{verbatim}  
\usepackage{multicol} 
\usepackage{enumerate}

\newcommand{\bbB}{\mathbb{B}}
\newcommand{\bbC}{\mathbb{C}}

\newcommand{\bbF}{\mathbb{F}}

\newcommand{\bbN}{\mathbb{N}}

\newcommand{\bbP}{\mathbb{P}}
\newcommand{\bbQ}{\mathbb{Q}}
\newcommand{\bbR}{\mathbb{R}}

\newcommand{\bbT}{\mathbb{T}}
\newcommand{\bbU}{\mathbb{U}}

\newcommand{\bbZ}{\mathbb{Z}}

\newcommand{\cB}{\mathcal{B}}

\newcommand{\cF}{\mathcal{F}}

\newcommand{\cJ}{\mathcal{J}}

\newcommand{\cO}{\mathcal{O}}
\newcommand{\cP}{\mathcal{P}}

\newcommand{\cS}{\mathcal{S}}
\newcommand{\cT}{\mathcal{T}}
\newcommand{\cU}{\mathcal{U}}

\newcommand{\fm}{\mathfrak{m}}

\newcommand{\fp}{\mathfrak{p}}

\renewcommand{\phi}{\varphi}

\providecommand{\abs}[1]{\left\lvert #1 \right\rvert}

\providecommand{\ar}{\rightarrow}

\DeclareMathOperator{\id}{id}

\DeclareMathOperator{\Spec}{Spec}

\DeclareMathOperator{\rk}{rk}

\providecommand{\tensor}{\otimes}

\DeclareMathOperator{\End}{End}

\DeclareMathOperator{\Gr}{Gr}

\DeclareMathOperator{\Mat}{Mat}

\DeclareMathOperator{\Fl}{Fl}

\DeclareMathOperator{\GL}{GL}
\DeclareMathOperator{\SL}{SL}

\DeclareMathOperator{\Cl}{Cl}
\DeclareMathOperator{\Pic}{Pic}

\DeclareMathOperator{\HH}{H}

\DeclareMathOperator{\St}{St}
\DeclareMathOperator{\lk}{lk}
\DeclareMathOperator{\str}{st}
\DeclareMathOperator{\sd}{sd}
\DeclareMathOperator{\Span}{span}

\newtheorem{thm}{Theorem}[section]
\newtheorem{cor}[thm]{Corollary}
\newtheorem{prop}[thm]{Proposition}
\newtheorem*{prop*}{Proposition}
\newtheorem{lem}[thm]{Lemma}

\newtheorem{quest}[thm]{Question}
\newtheorem*{IndProp}{Inductive Proposition}

\newtheorem{mainthm}{Theorem}

\theoremstyle{definition}
\newtheorem{defn}[thm]{Definition}

\newtheorem{exmp}[thm]{Example}

\numberwithin{claim}{thm}

\theoremstyle{remark}
\newtheorem{rem}[thm]{Remark}

\makeatletter
\let\c@equation\c@thm
\makeatother
\numberwithin{equation}{section}

\usepackage{hyperref}
\hypersetup{
	colorlinks = true,
	citecolor=Green,
	linkcolor=blue,
	urlcolor=violet,
}

\title{A Solomon-Tits theorem for Rings}

\author{Matthew Scalamandre}
\date{}

\begin{document}
	\begin{abstract}
		An analog of the Tits building is defined and studied for commutative rings. We prove a Solomon-Tits theorem when $R$ either satisfies a stable range condition, or is the ring of $S$-integers of a global field. We then define an analog of the Steinberg module of $R$, and study it both as a $\bbZ$-module and as a representation. We find the rank of Steinberg when $R$ is a finite ring, and compute the length of $\St_2(R)\otimes\bbQ$ as a $\GL_2(R)$-representation when $R$ is uniserial. As an application of these results, we produce a lower bound for the rank of the top-dimensional cohomology of principal congruence subgroups of nonprime level.
	\end{abstract}
	\maketitle
	\tableofcontents
	\section{Introduction}
	The Tits building $\cT_n(K)$ is a fundamental geometric object associated to $\GL_n(K)$ for a field $K$. The classical Solomon--Tits theorem says that this is homotopy equivalent to a wedge of $(n-2)$-spheres \cite{Solomon}. If $K$ is a number field or a finite field, the top-dimensional homology of $\cT_n(K)$ (called the \emph{Steinberg module}) arises in the study of arithmetic groups \cites{BS, MPP}. In trying to understand the cohomology of principal congruence subgroups of nonprime level, one is led to consider ``Tits buildings" associated to the rings $\bbZ/N\bbZ$. In this paper, we define a Tits complex which makes sense for an arbitrary commutative ring. We prove that $\cT_n(R)$ is highly connected when $R$ satisfies a stable range condition, or when $R$ is the ring of $S$-integers of a number ring. We then study its top-dimensional homology, both as a $\bbZ$-module and as a $\GL_n(R)$-representation. Using these results, we produce a lower bound for the rank of the top-dimensional cohomology of principal congruence subgroups of nonprime level.
	\subsection{Classical Theory}\;
	
	Let $K$ be a field. The classical \emph{Tits building} associated to $\GL_n K$ is the simplicial complex $\cT_n(K)$ whose vertex set is the set of nonzero proper linear subspaces of $K^n$, and which has a $k$-simplex for every flag:
	$$0\subsetneq V_0 \subsetneq\dots\subsetneq V_{k} \subsetneq K^n.$$
	This space admits a natural simplicial action of $\GL_n K$, and encodes fundamental information about the structure of this algebraic group. The homotopy type of $\cT_n(K)$ is determined by the Solomon-Tits theorem \cite{Solomon}, which states that it is homotopy equivalent to $\displaystyle \bigvee S^{n-2}.$
	
	It follows that $\cT_n(K)$ has only one nonzero reduced homology group, which is called the \emph{Steinberg module}: $$\St_n(K):= \tilde{\HH}_{n-2}(\cT_n(K);\bbZ).$$ This is naturally a representation of $\GL_n(K)$. There are explicit classes called \emph{apartment classes} which generate $\St_n(K)$. Let $\cB = \{b_1,\dots, b_n\}$ be a basis for $K^n$. The apartment $A_\cB$ corresponding to $\cB$ is the full subcomplex on the vertices of $\cT_n(K)$ corresponding to the subspaces $\Span\{b_i\mid i\in I\}$ where $I$ ranges over nonempty proper subsets of $\{1,\dots,n\}$. The apartment $A_\cB$ is isomorphic to the barycentrically subdivided boundary of the standard $(n-1)$-simplex, and is therefore homeomorphic to a sphere $S^{n-2}$. The fundamental class of this sphere defines an \emph{apartment class} in $\St_n(K)$.
	
	The Solomon-Tits theorem states that $\St_n(K)$ has a basis consisting of the apartment classes associated to the column vectors of strict upper-triangular matrices. In particular, this computes the dimension of the Steinberg module when $K$ is a finite field:
	$$\dim \St_n(\bbF_q) = q^{\binom{n}{2}}.$$
	The Steinberg module was first constructed in \cite{Steinberg} by Steinberg as a representation of $\GL_n(\bbF_q)$, using character-theoretic methods (the geometric construction came later, from Solomon \cite{Solomon}). Steinberg proved in \cite{Steinberg} that $\St_n(\bbF_q)\otimes \bbC$ is an irreducible $\GL_n(\bbF_q)$-representation. This result has recently been extended to infinite fields by Putman-Snowden \cite{SteinIrred}.
	
	For more about the history of the Steinberg module, we refer the reader to Humphreys' survey \cite{HumSurv}, and to Steinberg's commentary at the end of his \emph{Collected Papers} \cite{SteinbergSurvey}.
	
	\subsection{Cohomology of Arithmetic Groups}\;
	
	Let $K$ be a \emph{number field} -- a finite extension of $\bbQ$ -- and let $\cO_K$ be the ring of integers of $K$.	The Steinberg module of $K$ is closely related to the cohomology of arithmetic groups by a theorem of Borel and Serre \cite{BS}. This theorem states that $\St_n(K)$ is the \emph{rational dualizing module} of any finite index subgroup $\Gamma$ of $\SL_n(\cO_K)$. If $M$ is a $\bbQ[\Gamma]$-module, and $\nu$ is the cohomological dimension of $\SL_n(\cO_K)$, this means that  
		$$\HH^q(\Gamma; M) \cong \HH_{\nu-q}(\Gamma; \St_n(K)\otimes M).$$
	In particular, 
	$$\HH^{\nu}(\Gamma;\bbQ) \cong \HH_0(\Gamma; \St_n(K)) \cong (\St_n(K))_\Gamma.$$
	The right-hand side is the module of coinvariants, the largest $\Gamma$-invariant quotient of $\St_n(K)$. Using this framework, one can produce cohomology classes in $\HH^{\nu}(\Gamma;\bbQ)$ by understanding the structure of $\St_n(K)$ as a $\Gamma$-representation. 
	
	This has been profitably done in some cases. In \cite{LSz}, Lee-Szczarba compute that $(\St_n(K))_{\SL_n(\cO_K)} = 0$ when $\cO_K$ is a Euclidean domain with a multiplicative norm function. Ash-Rudolph \cite{AR} prove the stronger statement that, when $R$ is an arbitrary Euclidean domain and $F$ is its field of fractions, $\St_n(F)$ is in fact generated as an abelian group by apartment classes corresponding to $R$-bases of $F^n$. When this generation result holds, it can be easily seen that $(\St_n(F))_{\SL_n(R)} = 0$. In particular, these results imply that $$\HH^{\nu}(\SL_n(\cO_K);\bbQ) = 0$$ whenever $\cO_K$ is a Euclidean number ring.
	
	More recently, Church-Farb-Putman \cite{CFP} prove that $\St_n(K)$ is generated by $\cO_K$-apartment classes when $\cO_K$ is a PID and $K$ admits an embedding into $\bbR$. Conversely, they show that if $\cO_K$ is \emph{not} a PID, then $\St_n(K)$ is not generated by $\cO_K$-apartment classes. In fact, they show that
	$$\dim \HH^{\nu}(\SL_n(\cO_K);\bbQ) \geq (\abs{\Cl(\cO_K)}-1)^{n-1}.$$
	
	Church-Farb-Putman leave open the case of totally complex PIDs which are not Euclidean. By a result of Weinberger \cite{Weinberger}, if the Generalized Riemann hypothesis holds, the only examples of these are imaginary quadratic. Miller-Patzt-Wilson-Yasaki \cite{imQuad} show that $\St_n(K)$ is not generated by $\cO_K$-apartment classes when $\cO_K$ is an imaginary quadratic number ring that is not Euclidean.
	
	The above results all deal with the group $\SL_n(\cO_K)$. Borel--Serre's theorem also applies to finite-index subgroups of $\SL_n(\cO_K)$. Recall the principal congruence subgroups:
	
	\begin{defn}
		Let $R$ be a ring, and $I$ an ideal. Then, the \emph{principal congruence subgroup} $\Gamma_n(I)$ of $\GL_n(R)$ is the kernel of the natural map $\GL_n(R)\to \GL_n(R/I)$. Equivalently, $$\Gamma_n(I) = \{A\in \GL_n(R)\mid A\equiv \id \pmod{I}\}.$$
	\end{defn}
	
	Define the \emph{level $m$ principal congruence subgroup} of $\SL_n(\bbZ)$ to be $$\Gamma_n(m):=\Gamma_n(m\bbZ).$$
		
	The Tits buildings associated to finite fields are relevant to the top-dimensional cohomology of principal congruence subgroups of prime level $p$. To find a lower bound on the dimension of this cohomology using Borel--Serre's theorem, we want to find a large $\Gamma_n(p)$-invariant quotient of $\St_n(\bbQ)$. A natural candidate is the group $\tilde{\HH}_{n-2}(\cT_n(\bbQ)/\Gamma_n(p),\bbZ)$. We have a diagram $$\xymatrix{\cT_n(\bbQ) \ar[rr] \ar[dr] & & \cT_n(\bbF_p) \\ & \cT_n(\bbQ)/\Gamma_n(p) \ar[ur] & }$$ where the map $\cT_n(\bbQ)\to \cT_n(\bbF_p)$ is defined on vertices by sending a subspace $V\subsetneq \bbQ^n$ to the image of $(V\cap \bbZ^n)$ in $\bbF_p^n$. In this case, Miller-Patzt-Putman \cite{MPP} use $\cT_n(\bbF_p)$ to analyze the topology of $\cT_n(\bbQ)/\Gamma_n(p)$, and conclude that the map 
	$$\St_n(\bbQ)\to \tilde{\HH}_{n-2}(\cT_n(\bbQ)/\Gamma_n(p),\bbZ)$$
	is surjective but usually not injective. We would like this framework to work for general levels $m$, but to do so, we first need to understand what is meant by $\cT_n(\bbZ/m\bbZ)$.
	\begin{rem}
		In \cite{Schwermer}, Schwermer uses the theory of automorphic forms to provide a lower bound for $\dim \HH^\nu(\Gamma(p^k),\bbQ)$, for odd $p$, such that $p^k>5$. This is a stronger lower bound than what is found by Miller-Patzt-Putman.
	\end{rem}
		
	\subsection{The Tits Complex over a Ring}\;
	
	To this end, we define a \emph{Tits complex} $\cT_n(R)$ for a commutative ring $R$:
	\begin{defn}
		Let $\cT_n(R)$ be the simplicial complex whose vertex set consists of those direct summands $V$ of $R^n$ such that both $V$ and $R^n/V$ are free, and which has a $k$-simplex for every flag
		$$0\subsetneq V_0\subsetneq\dots\subsetneq V_{k}\subsetneq R^n$$
		that is a subflag of some complete flag.
	\end{defn}
	
	The vertex set is exactly the set of length 1 flags which can be extended to a complete flag (cf. Proposition \ref{prop:goodFlags}).  This definition agrees with the classical definition when $R$ is a field, and is functorial in $R$. It is additionally always a \emph{pure} complex of dimension $(n-2)$: every simplex is contained in a simplex of maximum dimension. 
	
	\begin{rem}\label{rem:OKvsK}
		We choose not to call this complex a ``Tits building". If $R$ is not a field, it does not appear to be the Br\^uhat-Tits building associated to any $BN$-pair on $\GL_n(R)$. However, in some interesting cases, this complex is in fact a building. 
		Let $\cO_K$ be the ring of integers of a number field $K$. Then, $$\cT_n(\cO_K)\subseteq\cT_n(K).$$ These complexes are equal exactly when $\cO_K$ is a PID (cf.\ Theorem \ref{thm:localTits}). In particular, $\cT_n(\bbZ)$ is equal to the Tits building $\cT_n(\bbQ)$.
	\end{rem}
	
	\begin{rem}
		Motivated by the rank filtration in algebraic K-theory, Rognes defines a similar complex in the unpublished \cite{RognesUnp}. This coincides with our Tits complex when $R = \bbZ/p^n\bbZ$. 
	\end{rem}
	
	Our main result determines the homotopy type of this complex for two families of rings, whose definitions we recall: 
	\begin{defn}
		A ring $R$ satisfies $SR_2$ if, whenever $a,b\in R$ together generate the unit ideal, there is some $t\in R$ such that $a-tb$ is a unit.
	\end{defn}
	
	The property $SR_2$ is one of Bass' stable range conditions, which we discuss in \S\ref{section:Bass}. All local rings satisfy $SR_2$, as do all Artinian rings and thus all finite rings (cf. Proposition \ref{prop:semilocalSR2}). In particular, if $\cO$ is a number ring and $I$ is a nonzero ideal, then $\cO/I$ satisfies $SR_2$. We are especially interested in these rings for their applications to the cohomology of arithmetic groups. 
	
	\begin{defn}
		Let $K$ be a number field. Let $\Omega$ be the set of places of $K$, thought of as equivalence classes of nontrivial valuations on $K$. The \emph{infinite places} are the Archimedian valuations corresponding to embeddings of $K$ into $\bbR$ or $\bbC$. The remainder correspond to embeddings into $p$-adic fields. We write $S_\infty$ for the set of infinite places. Given a finite $S\subseteq \Omega$ containing $S_\infty$, we define
		$$\cO_S = \{ x\in K \mid \abs{x}_v\leq 1 \text{ for all } v\in \Omega\setminus S\}.$$
		We call these rings \emph{Dedekind domains of arithmetic type}. These include all number rings (when $S=S_\infty$) as well as certain localizations of number rings. 
	\end{defn}

	Our main theorem states that $\cT_n(R)$ is highly connected. In fact, we prove more. A simplicial complex $X$ is \emph{$n$-spherical} if $X$ is $n$-dimensional and $(n-1)$-connected. The \emph{star} of a $k$-simplex $\sigma$ (written $\str \sigma$) is the closure of the subcomplex of $X$ consisting of all simplices containing $\sigma$ as a face. This is always contractible. The \emph{link} of $\sigma$ (written $\lk \sigma$) is the subcomplex of $\str\sigma$ consisting of every face which is disjoint from $\sigma$. The star of a simplex $\sigma$ is the simplicial join of $\sigma$ and $\lk \sigma$.
	\begin{defn}\label{def:CM}
		A simplicial complex $X$ is \emph{Cohen-Macaulay of dimension $n$} if it is $n$-spherical and if, for every $k$-simplex $\sigma$, we have that $\lk \sigma$ is $(n-k-1)$-spherical.
	\end{defn}
	Our first main theorem is:
	\begin{mainthm}\label{Connectivity}
		Let $R$ be a ring satisfying $SR_2$, or a Dedekind domain of arithmetic type. Then, the Tits complex $\cT_n(R)$ is Cohen-Macaulay of dimension $n-2$. In particular, it is $(n-3)$-connected.
	\end{mainthm}
	This reduces to the classical Solomon--Tits theorem when $R$ is a field. Our proof of this theorem is not similar to any proof of the classical Solomon-Tits theorem known to the author, and occupies all of \S4. 
	
	\begin{rem}
		Using different methods, Rognes proves Theorem \ref{Connectivity} for $R=\bbZ/p^n\bbZ$ in \cite{RognesUnp}. It seems difficult to extend his methods to address larger families of rings.
	\end{rem}
	
	We define the \emph{Steinberg module of $R$} $$\St_n(R):= \tilde{\HH}_{n-2}(\cT_n(R);\bbZ),$$ and define \emph{apartment classes} as follows.
	
	\begin{defn}\label{defn:aptclass}
		Let $\{v_1,\dots v_n\}$ be a free basis of $R^n$. Consider the poset $\cP([n])$ of proper nonempty subsets of $\{1,\dots, n\}$ under containment. The nerve of this poset (cf. Definition \ref{defn:nerve}) is the barycentrically subdivided boundary of the $(n-1)$-simplex $\sd\partial\Delta^{n-1}$, which is homeomorphic to $S^{n-2}$. We have a simplicial map $$f\colon\sd\partial\Delta^{n-1}\to \cT_n(R)$$ sending a subset $\{i_1, \dots, i_k\}$ to $\Span\{v_{i_1}, \dots, v_{i_k}\}$. The \emph{apartment class} $\left[\begin{array}{c|c|c}v_1 & \dots & v_n\end{array}\right]$ is then to be the pushforward of the fundamental class $f_*[S^{n-2}]\in \tilde{\HH}_{n-2}(\cT_n(R);\bbZ)$.
	\end{defn}
	
	In \S5, we prove that these generate the Steinberg module:
	\begin{mainthm}\label{mainthm:AptsGen}
		If $R$ satisfies $SR_2$, or is a Dedekind domain of arithmetic type with a non-complex embedding, then $\St_n(R)$ is generated by apartment classes.
	\end{mainthm}
	By Remark \ref{rem:OKvsK}, this is essentially already known for Dedekind domains of arithmetic type which are PIDs, but is otherwise new. When a number ring $\cO_K$ is not a PID, the results of Church--Farb--Putman \cite{CFP} imply that $\St_n(\cO_K)$ cannot be isomorphic to $\St_n(K)$. In fact, $\St_n(\cO_K)$ is exactly the $\SL_n(\cO_K)$-subrepresentation of $\St_n(K)$ generated by $\cO_K$-apartment classes. The requirement of a non-complex embedding cannot be dropped, by the main theorem of Miller--Patzt--Wilson--Yasaki \cite{imQuad}.
	
	It is then natural to ask if the rest of the Solomon--Tits theorem holds: namely, does $\St_n(R)$ have a basis consisting of upper-triangular apartment classes? The answer is negative:
	\begin{mainthm}\label{mainthm:UTApts}
		Let $R$ be a commutative ring. The upper-triangular apartment classes are linearly independent in $\St_n(R)$, but do not span unless $R$ is a field.
	\end{mainthm}
	However, we still produce a recursive formula for the dimension of $\St_n(R)$ over a finite ring.
	\begin{defn}
		Let $n\geq k\geq 0$. The \emph{Grassmannian} $\Gr_k^n(R)$ is the set of rank $k$ free and cofree direct summands of $R^n$. 
	\end{defn}
	This is a transitive $\GL_n(R)$-set, and its cardinality can be computed for finite rings (cf. Corollary \ref{flagTrans} and Proposition \ref{GrassSize}). Notice that $\abs{\Gr_0^n(R)}=1$. We prove:
	\begin{mainthm}\label{recdim}
		Let $R$ a finite ring, and, for $n\geq 1$, let $d_n = \rk \St_n(R)$. Define $d_0 = 1$. We have the following recursive formula:
		$$d_0 = 1; \qquad \qquad d_n =\sum_{i=1}^n (-1)^{i-1} \abs{\Gr_{n-i}^n(R)} d_{n-i}\ \ \text{ for } n\geq 1.$$
	\end{mainthm}
	The author is not aware of a closed form for this expression (except, of course, when $R$ is a finite field). We do find that, when $\abs{R}$ is sufficiently large, the dimension of $\St_n(R)$ is not much bigger than $\abs{R}^{\binom{n}{2}}$. So in these cases, the upper-triangular apartment classes form ``most" of a basis for $\St_n(R)$. The form of Theorem \ref{recdim} is also amenable to computer implementation -- see Table \ref{table:SteinDim}.
	
	By the functoriality of the Tits complex, we have a map $$\St_n(\bbQ)\cong \St_n(\bbZ)\to \St_n(\bbZ/m\bbZ).$$ By Theorem \ref{mainthm:AptsGen}, this map is surjective, and so $\St_n(\bbZ/m\bbZ)$ is a quotient of $\St_n(\bbQ)$. This is a map of $\SL_n(\bbZ)$-representations, and $\St_n(\bbZ/m\bbZ)$ is $\Gamma_n(m)$-invariant. Borel-Serre duality then implies 
	$$\dim\HH^{\binom{n}{2}}(\Gamma_n(m);\bbQ)\geq \rk \St_n(\bbZ/m\bbZ).$$
	A similar argument works for a larger family of number rings:
	\begin{mainthm}\label{mainthm:cohomology}
		Let $\cO$ be a number ring, and let $I$ be a nonzero ideal of $\cO$. Let $\nu$ be the cohomological dimension of $\SL_n(\cO)$. Then, 
		\[ \dim\HH^{\nu}(\Gamma_n(I);\bbQ)\geq \rk \St_n(\cO/I). \]
	\end{mainthm}
	
	Therefore, Theorem \ref{recdim} computes a lower bound on the top-dimensional cohomology of $\Gamma_n(I)$. This result is new, except when $\cO = \bbZ$ and $I$ is of the form $p^k\bbZ$ for some prime $p$. In this case, our bound on $\dim \HH^{\binom{n}{2}}(\Gamma_n(p^k);\bbQ)$ is weaker than the bounds in \cite{MPP} or \cite{Schwermer}.
	\begin{rem}
		We do not expect that this bound is tight in general. Indeed, we expect that the methods of \cite{MPP} could be applied to produce a stronger lower bound. The author intends to take up this question in a subsequent paper.
	\end{rem}
	In \S6, we investigate the representation-theoretic structure of $\St_n(R)\otimes \bbQ$ (which we write as $\St_n^\bbQ(R)$). If $R$ is not a field, then $\St^\bbQ_n(R)$ is never irreducible. Indeed, for any ideal $I$ we have a nonzero map
	$$\St^\bbQ_n(R)\to \St^\bbQ_n(R/I),$$
	whose kernel is a proper subrepresentation (c.f. Theorem \ref{SteinReducible}). When $R$ is a finite ring, we can identify $\St^\bbQ_n(R/I)$ with the $\Gamma_n(I)$-invariants of $\St^\bbQ_n(R)$:
	\begin{mainthm}\label{mainthm:coinvariants}
		Let $R$ be a finite ring, $\cJ$ its Jacobson radical, and let $I$ be an ideal of $R$ contained in $\cJ$. Then
		$$\St_n^\bbQ(R)^{\Gamma(I)} \cong \St_n^\bbQ(R/I)$$ as $\GL_n(R)$-representations.
	\end{mainthm}
	Notice that, in this case, the $\Gamma_n(I)$-coinvariants agree with the $\Gamma_n(I)$-invariants. To conclude, we ask the following question:
	\begin{quest}
		For a finite local ring $R$, what is the length of $\St^\bbQ_n(R)$ as a $\GL_n(R)$-representation?
	\end{quest}
	For quotients of Dedekind domains, we answer the question when $n=2$.
	\begin{mainthm}\label{mainthm:length2}
		Let $\cO$ be a Dedekind domain, and $\fp\leq \cO$ a prime ideal. Then, $\St_2^\bbQ(\cO/\fp^k)$ has length exactly $k$.
	\end{mainthm}
	
	This is the minimum length allowed by Theorem \ref{mainthm:coinvariants}.
	\begin{table}
		\begin{footnotesize}
			\begin{tabular}{|c||r|r|r|r|r|}
				\hline
				& \multicolumn{5}{|c|}{$d$}\\
				\hline $n$ & 4 & 6 & 8 & 9 & 10 \\
				\hline 1& 1 & 1 & 1 & 1 & 1 \\
				2& 5 & 11 & 11 & 11 & 17 \\
				3& 113 & 911 & 1121 & 1171 & 3473  \\
				4& 10879 & 497149 & 978559 & 1149929 & 7649589  \\
				5& 4324129 & 1696007149 & 7061119489 & 10247219929 & 174326656989\\
				6 & 6984271295 & 35372169269639 & 414187232163839 & 824092678295459 & 40378418645294393 \\
				\hline
			\end{tabular}
		\end{footnotesize}
		\caption{Rank of $\St_n(\bbZ/d\bbZ)$, for composite $d\leq 10$.}
		\label{table:SteinDim}
	\end{table}
	\subsubsection{Acknowledgments}
	I would like to thank my advisor Andrew Putman for introducing me to this problem, and for his guidance during every stage of this project. I would also like to thank Samuel Evens for helpful conversations about the representation theory of $\GL_n(\bbF_q)$. I would like to thank Alexander Kupers and Jeremy Miller for pointing out some useful references. 
	
	This article is based upon work supported by the National Science Foundation under Grant Number DMS-1547292.
	\section{Algebraic Preliminaries}
	Here, we collect some algebraic tools and definitions that will be used throughout the paper.
	Throughout the rest of this paper, $R$ will always be a commutative ring. 
	
	\subsection{Flags of Free Summands}\label{Section:flags}\;
	
	We study the collections of flags which are important to the construction of $\cT_n(R)$.
	\begin{defn}
		A \emph{flag} in $R^n$ is a chain of free direct summands $$0\subsetneq V_1\subsetneq\dots\subsetneq V_k\subsetneq R^n$$
		For a partition $\lambda = (\lambda_1,\dots,\lambda_{k+1})$ of $n$, we say that a flag has \emph{type} $\lambda$ if $$\rk V_j = \sum_{i=1}^j \lambda_i \text{ for }j\leq k.$$
		A flag is \emph{complete} if $k=n-1$, in which case $V_i$ has rank $i$ (equivalently, these are the flags of type $(1,1,\dots,1)$).
	\end{defn}
	Notice that when $K$ is a field, any flag in $K^n$ is a subflag of a complete flag. This makes $\cT_n(K)$ a pure simplicial complex: every simplex is contained in a simplex of maximum dimension. This is a desirable property for the Tits complex to have, since every Cohen-Macaulay complex is pure. We make the following definition:
	\begin{defn}
		We say that a flag $\cF = \{0\subsetneq V_1\subsetneq\dots\subsetneq V_k\subsetneq R^n\}$ is \emph{good} if it can be completed, i.e., if we can find a complete flag $\{0\subsetneq W_1\subsetneq\dots \subsetneq W_{n-1}\subsetneq R^n\}$ and indices $j_i$ with $V_i = W_{j_i}$ for all $i\leq k$. 
	\end{defn}
	
	One can produce good flags using the following procedure:
	\begin{defn}
		Let $\lambda$ be a partition of $n$ and let $(v_1,\dots, v_n)$ be an ordered basis of $R^n$. We write $\Span_\lambda (v_1,\dots, v_n)$ for the flag \begin{align*}
			0\subsetneq \Span\{v_1,\dots, v_{\lambda_1}\} \subsetneq \Span\{v_1&,\dots, v_{\lambda_1+\lambda_2}\}\subsetneq\dots \\ \dots &\subsetneq \Span\{v_1,\dots, v_{\lambda_1+\dots+\lambda_{k-1}}\}\subsetneq R^n
		\end{align*}
		If $\cF$ is a flag such that $\cF = \Span_\lambda(v_1,\dots, v_n)$, we say that $(v_1,\dots, v_n)$ is a \emph{flag-basis} for $\cF$.
	\end{defn}
	It turns out that all good flags arise in this manner. To prove this, we need a definition and a lemma:
	\begin{defn}
		Let $M$ be an $R$-module and $N$ be a direct summand of $M$. We say that $N$ is \emph{cofree} in $M$ if it has a free complement, i.e.\ there is a free direct summand $N'$ of $M$ with $ N\oplus N'= M$. Equivalently, $N$ is cofree if $M/N$ is free.
	\end{defn}
	
	\begin{lem}\label{codim1free}
		Any free summand of $R^n$ of rank $(n-1)$ is cofree.
	\end{lem}
	\begin{proof}
		Let $N$ be such a free summand, and let $L$ be its complement. Consider the determinant construction: 
		$$\textstyle R \cong \bigwedge^n R^n \cong \bigwedge^n (N \oplus L) \cong \displaystyle \bigoplus_{i = 0}^n\textstyle (\bigwedge^i N)\tensor (\bigwedge^{n-i} L).$$ 
		We will show that $\bigwedge^j L = 0$ when $j\geq 2$, so that each term of the right hand side is zero except for $i = n-1$. It follows that the right hand side is $L\otimes R \cong L$, and so $L\cong R$ as required.
		
		Let $\fp$ be a prime of $R$. By exactness of localization and the fact that all projective modules over a local ring are free, $L_\fp$ is free of rank 1 over $R_\fp$.
		If $j\geq 2$, we see that $(\bigwedge^j L)_\fp \cong \bigwedge^j (L_\fp) = 0$. Since $\fp$ is arbitrary, $\bigwedge^j L$ has empty support, and so is the zero module. 
	\end{proof}
	\begin{prop}\label{prop:goodFlags}
		For a flag $\cF = \{V_1\subsetneq\dots\subsetneq V_k\}$ of type $\lambda$, the following are equivalent:
		\begin{enumerate}
			\item $\cF$ is good.
			\item $\cF$ admits a flag-basis.
			\item $V_i$ is cofree in $V_{i+1}$ for all $i<k$, and $V_k$ is cofree in $R^n$.
		\end{enumerate}
	\end{prop}
	\begin{proof}
		$(2) \implies (1)$ and $(2)\implies (3)$ are both immediate. 
		
		$(3)\implies (2)$: Build the flag-basis inductively: let the first $\lambda_1$ vectors be any basis for $V_1$, the next $\lambda_2$ be any basis for a free complement of $V_1$ inside $V_2$, and so on.
		
		$(1)\implies (2)$: Let $\cF'$ be a complete flag containing $\cF$. By Lemma \ref{codim1free}, $\cF'$ satisfies condition (3), and admits a flag basis $(v_1,\dots, v_n)$. Therefore: $$\cF' = \Span_{(1,\dots,1)} (v_1,\dots, v_n).$$ Then, by the definition of flag-bases, $\cF = \Span_\lambda (v_1,\dots, v_n)$.
	\end{proof}
	Notice that $\GL_n(R)$ acts transitively on the set of flag-bases, and that if $u\in R^\times$, the matrix $\begin{bmatrix} u & 0 \\ 0 & \id_{n-1} \end{bmatrix}$ stabilizes the standard flag of type $\lambda$. It follows that:
	\begin{cor}\label{flagTrans}
		$\GL_n(R)$ and $\SL_n(R)$ each act transitively on the set $\Fl_\lambda$ of good flags of type $\lambda$.
	\end{cor}
	\begin{exmp}[A non-good flag]\label{exmp:nongood}
		Set $R = \bbR[x,y,z]/(x^2+y^2+z^2 -1)$, the ring of polynomial functions on $S^2$, and consider the flag:
		$$0\subsetneq R(xe_1+ye_2+ze_3) \subsetneq Re_1\oplus Re_2\oplus Re_3 \subsetneq R^4.$$
		Each term in this flag is free and cofree in $R^4$. Every projective module over $R$ induces a vector bundle over $S^2$. One can check that the flag above induces the flag of bundles:
		$$0\subsetneq \nu(S^2)\subsetneq T\bbR^3|_{S^2}\subsetneq T\bbR^4|_{S^2}$$
		where $\nu(S^2)\subsetneq T\bbR^3|_{S^2}$ is the standard inclusion of the normal bundle of $S^2$ into $T\bbR^3$. If the original flag of $R$-modules is a good flag, then $\nu(S^2)$ will be cofree in $T\bbR^3|_{S^2}$. This is well-known to be false.
	\end{exmp}

	\subsection{Bass' Stable Range Condition}\label{section:Bass}\;
	
	We discuss a family of conditions on a ring which generalize the condition $SR_2$ defined in the introduction.
	
	\begin{defn}
		Let $R$ be a ring, and let $n\geq 2$. We say that $R$ satisfies \emph{Bass' stable range condition} $SR_n$ if for every subset $\{r_1,\dots, r_n\}$ of $R$ generating the unit ideal, there is a sequence of elements $t_1,\dots, t_{n-1}$ such that $\{r_1-t_1r_n, \dots, r_{n-1}-t_{n-1}r_n\}$ also generates the unit ideal. 
	\end{defn}
	
	Notice that a ring satisfying $SR_n$ necessarily satisfies $SR_{n+1}$. These conditions are difficult to verify in general, but the following theorem of Bass gives many examples:
	
	\begin{thm}[Bass Cancellation Theorem, {\cite[Theorem V.3.5]{BassAlgKThry}}]
		If $R$ is a commutative ring of Krull dimension $d$, then $R$ satisfies $SR_{d+2}$.
	\end{thm}
	
	Artinian rings have Krull dimension $0$, and thus satisfy $SR_2$. Number rings satisfy $SR_3$, but not $SR_2$. For example, in $\bbZ$, the set $\{5,7\}$ generates the unit ideal, but there is no choice of $t\in \bbZ$ such that $5-7t = \pm 1$. 
	
	\begin{defn}
		Let $R$ be commutative ring, and $\cJ$ its Jacobson radical. We say $R$ is \emph{semilocal} if $R/\cJ$ is a product of fields. Equivalently, $R$ has finitely many maximal ideals.
	\end{defn}
	\begin{prop}\label{prop:semilocalSR2}
		Every semilocal ring satisfies $SR_2$.
	\end{prop}
	\begin{proof}
		First, we prove the proposition in the special case where $R\cong F_1\times \dots\times F_k$ is a product of fields. Let $\{a,b\}$ generate the unit ideal, where $a = (a_1,\dots, a_k)$ and $b = (b_1,\dots, b_k)$. Since $\{a,b\}$ generates the unit ideal, for any $i$ where $a_i=0$, $b_i\ne 0$. Set $t = (t_1,\dots, t_k)$ where $t_i = 1$ if $a_i=0$, and $t_i = 0$ else. Then, $a-tb$ is nonzero everywhere, and thus a unit of $R$.
		
		We now consider the case where $R$ is a general semilocal ring. If $\cJ$ is the Jacobson radical, then $R/\cJ$ is a product of fields. Let $\{a,b\}$ generate the unit ideal of $R$. Then, their images $\{\overline{a},\overline{b}\}$ generate the unit ideal of $R/\cJ$. Choose $\overline{t}$ as above, so $\overline{a}-\overline{t}\overline{b}$ is a unit $\overline{u}$ of $R/\cJ$. Let $t$ be some lift of $\overline{t}$, and let $u^{-1}$ be a lift of $\overline{u}^{-1}$. Then, $(a-tb)u^{-1}\in 1 + \cJ$, which consists of units, so $a-tb$ is itself a unit.
	\end{proof}
	
	We state an important application of stable range conditions. Recall that a finitely-generated projective $R$-module $P$ has \emph{constant rank} $d$ if, for any maximal ideal $\fm$, the vector space $P\otimes R/\fm$ has dimension $d$.
	\begin{thm}\label{BassCancel}
		Let $R$ satisfy $SR_k$, and let $P$ be a finitely generated $R$-module of constant rank $d\geq k-1$. If $P$ is stably free, then it is free.
	\end{thm}
	This is essentially due to Bass and Serre. See \cite[Lemma 8.9]{TwistedStability} for a clear exposition.
	
	\subsection{A Presentation for $\SL_n(R)$}\label{Ktheory}\;
	
	The following material is classical, and is drawn from \cite{Milnor} and \cite{KBook}.
	
	To handle the fundamental group of $\cT_n(R)$, we will need to produce and understand a presentation of $\SL_n(R)$.	
	If $i\ne j$ and $\lambda\in R$, the \emph{elementary matrix} $e_{ij}(\lambda)$ is the matrix which has $\lambda$ in the $ij$th position, and which agrees with the identity matrix elsewhere. We denote by $E_n(R)$ the subgroup of $\SL_n(R)$ generated by elementary matrices.
	
	If $R$ is a field or Euclidean domain, then $\SL_n(R)$ is generated by elementary matrices, and so $\SL_n(R) = E_n(R)$. For general rings, this is not the case, and the deficit is measured by the unstable K-theory group $SK_1(n,R)$ (although we will not need this technology here).	We briefly survey the state of affairs for semilocal rings and number rings:
	
	\begin{thm}\label{thm:SR2SK1}
		If $R$ satisfies $SR_2$, then $\SL_n(R)$ is generated by elementary matrices for all $n$.
	\end{thm}
	\begin{proof}
		By Theorem 4.2.5 in \cite{HahnO'Meara}, $\GL_n(R)=\GL_1(R)E_n(R)$. If $g\in \SL_n(R)$, use this to write $g = ue$, where $u\in\GL_1(R)\cong R^\times$ and $e\in E_n(R)$. Taking determinants, we see that $u=1$.
	\end{proof}
	
	\begin{thm}\label{E2}
		Let $K$ be a number field, $S$ be a set of places containing all infinite places, and $\cO_S$ be be the associated Dedekind domain of arithmetic type.
		\begin{itemize}
			\item If $\cO_S$ is Euclidean, then $\SL_n(\cO_S)$ is generated by elementary matrices.
			\item $($Bass-Milnor-Serre \cite{BassMilnorSerre}$)$ If $n\geq 3$, then $\SL_n(\cO_S)$ is generated by elementary matrices.
			\item $($Vaserstein \cite{Vaserstein}$)$ If $\abs{S}\geq 2$, then $\SL_2(\cO_S)$ is generated by elementary matrices.
			\item $($Cohn \cite{Cohn}, Nica \cite{Nica}\footnote{See Sheydvasser \cite{Sheydvasser} for a corrigendum.}$)$ If $K$ is imaginary quadratic, then $\SL_2(\cO_K)$ is generated by elementary matrices if and only if $\cO_K$ is Euclidean. In the case where this does not hold, $E_2(\cO_K)$ is not normal in $SL_2(\cO_K)$, and has infinite index.
		\end{itemize}
	\end{thm}
	
	We now consider relations for $E_n(R)$. Some relations are purely structural, and are recorded in the Steinberg group:
	
	\begin{defn}
		Let $n\geq 3$. We define the \emph{Steinberg group} $S_n(R)$ to be the group generated by formal symbols $x_{ij}(\lambda)$, where $1\leq i,j\leq n$, $i\ne j$, $\lambda\in R$, subject to the following relations:
		\begin{itemize}
			\item $x_{ij}(\lambda)x_{ij}(\mu) = x_{ij}(\lambda+\mu)$
			\item $[x_{ij}(\lambda),x_{k\ell}(\mu)] = \begin{cases}
				1 & j\ne k\text{ and } i\ne \ell, \\
				x_{i\ell}(\lambda\mu) & j= k\text{ and } i\ne \ell, \\
				x_{kj}(-\mu\lambda) & j\ne k\text{ and } i= \ell.
			\end{cases}$
		\end{itemize}
	\end{defn}
	
	The elementary matrices satisfy the defining relations of the Steinberg group, so there is a surjective map $S_n(R)\to E_n(R)$ sending $x_{ij}(\lambda)$ to $e_{ij}(\lambda)$.
	
	\begin{defn}
		The kernel of the map $S_n(R)\to E_n(R)$ is written $K_2(n,R)$.
	\end{defn}
	This is a subtle algebraic invariant of $R$. Fortunately for us, it stabilizes.	
	There is a natural inclusion $E_n(R)\hookrightarrow E_{n+1}(R)$, sending a matrix $A$ to $\begin{bmatrix} A & 0 \\ 0 & 1 \end{bmatrix}$. Similarly, we have a map $S_n(R)\hookrightarrow S_{n+1}(R)$ sending $x_{ij}(\lambda)$ to $x_{ij}(\lambda)$. Using these inclusions, we produce a natural map $K_2(n,R)\to K_2(n+1,R)$.
	
	\begin{thm}[Bass, Vaserstein, Dennis, Suslin, Tulinbayev, van der Kallen \cite{SurjStab}]\label{surjStab}
		Let $R$ satisfy $SR_k$. Then for $n\geq k$, the natural maps $$K_2(n,R)\to K_2(n+1,R)$$ and $$K_2(n,R)\to K_2(R)$$ are surjective, and when $n\geq k+1$ they are isomorphisms. 
	\end{thm}
	
	If $R$ satisfies $SR_3$, it follows that $K_2(3,R)$ surjects onto $K_2(n,R)$ for $n\geq 3$. The output of this theorem is the following fact, which we need in the sequel:
	
	\begin{thm}\label{thm:SLnPres}
		If $n\geq 3$ and $R$ is such that the natural map $K_2(3,R)\to K_2(n,R)$ is surjective, then $E_n(R)$ is isomorphic to the group generated by symbols $x_{ij}(\lambda)$, where $1\leq i,j\leq n$, $i\ne j$, $\lambda\in R$, subject to the relations:
		\begin{itemize}
			\item $x_{ij}(\lambda)x_{ij}(\mu) = x_{ij}(\lambda+\mu)$
			\item $[x_{ij}(\lambda),x_{k\ell}(\mu)] = \begin{cases}
				1 & j\ne k\text{ and } i\ne \ell, \\
				x_{i\ell}(\lambda\mu) & j= k\text{ and } i\ne \ell, \\
				x_{kj}(-\mu\lambda) & j\ne k\text{ and } i= \ell.
			\end{cases}$
			\item Any additional relations present in $E_3(R)$.
		\end{itemize}
	\end{thm}

	\section{Simplicial Complexes}
	
	In this section, we present tools to analyze the topology of simplicial complexes, and then apply these tools to analyze the Tits complex.
	
	\subsection{Topology of Posets}\label{simptools}\;
	
	We discuss a relationship between posets and simplicial complexes. We have the following definition:
	\begin{defn}\label{defn:nerve}
		Let $P$ be a poset. The \emph{nerve} of $P$ is the simplicial complex $N(P)$ whose vertices are elements of $P$, and with a $k$-simplex $\sigma$ for every set of vertices $\{p_0,\dots, p_k\}$ such that $$p_0 < p_1 <\dots< p_{k}.$$ The nerve is a functor from the category of posets to the category of spaces.
	\end{defn}
	\begin{exmp}
		Consider the poset $\{a_1,a_2,b_1,b_2\}$, ordered such that $a_i \leq b_j$ for all $i,j$, and no other comparisons hold. The nerve has four vertices, and four edges: $\{a_1,b_1\}, \{a_1,b_2\}, \{a_2,b_1\},\{a_2,b_2\}$. The resulting space is homeomorphic to $S^1$:
		$$\xymatrix{ [a_1] \ar@{-}[d] \ar@{-}[r] & [b_1] \ar@{-}[d] \\ [b_2] \ar@{-}[r] & [a_2]}$$
	\end{exmp}
	Every simplicial complex is homeomorphic to the nerve of a poset:
	\begin{exmp}\label{subdiv}
		Let $X$ be a simplicial complex. Let $P_X$ be the poset consisting of all simplices of $X$, ordered by inclusion. The space $N(P_X)$ is isomorphic to the barycentric subdivision of $X$, and as such is homeomorphic to $X$.
	\end{exmp}
	\begin{defn}
		Let $X$ be a poset, and let $x\in X$. The \emph{height of $x$}, written $h(x)$, is the maximum $h$ such that there is a chain in $X$ ending at $x$:
		$$p_0<p_1<\dots<p_{h} = x$$
		If there is such a chain for all $h$, then we write $h(x)=\infty$.
	\end{defn}
	The \emph{dimension} of a poset $X$ is $\displaystyle \dim X = \max_{x\in X} h(x)$; this agrees with the dimension of the nerve. 
	If $P$ is a poset, we say that it is $n$-spherical (resp. Cohen-Macaulay of dimension $n$) if $N(P)$ is. We have the following characterizations of the star and link. Let $\sigma = \{p_0<p_1<\dots < p_{k}\}$. Then, 
	\begin{align*}
		\str_P \sigma &= N\left(\{x \in P \mid x\text{ is comparable to }p_i\ \forall i\}\right)\\
	\lk_P \sigma &= N\left(\{x\in P\mid x \text{ extends the chain } p_0<p_1<\dots < p_{k}\}\right)
	\end{align*}
	Additionally, we define the following subposets. If $x\in X$, then:
	\begin{itemize}
		\item $X_{<x} := \{x'\in X \mid x' < x \}$
		\item $X_{>x} := \{x'\in X \mid x' > x \}$
	\end{itemize}
	
	We have the following poset-level condition for a complex to be Cohen-Macaulay, due to Quillen:
	\begin{prop}[Quillen, {\cite[Proposition 8.6]{Quillen}}]\label{prop:QuillenCM}
		A poset $X$ is Cohen-Macaulay of dimension $n$ if and only if the following all hold:
		\begin{itemize}
			\item $X$ is $n$-spherical.
			\item For all $x\in X$, the poset $X_{<x}$ is $(h(x)-1)$-spherical.
			\item For all $x\in X$, the poset $X_{>x}$ is $(n-h(x)-1)$-spherical.
			\item For all $x,x'\in X$ with $x<x'$, the poset $(X_{>x}\cap X_{<x'})$ is $(h(x')-h(x)-2)$-spherical.
		\end{itemize}
	\end{prop}

	Let $f:X\to Y$ be a map of posets. Then, for any $y\in Y$, we define:
	$$f/y := \{x\in X\mid f(x)\leq y\}.$$
	
	When we speak of the homology of a poset, we refer to the homology of its nerve. The following theorem of Quillen helps us compute this homology.
	
	\begin{thm}[Quillen, {\cite[Theorem 9.1]{Quillen}}]\label{Quillen}
		Let $f\colon X\to Y$ be a map of posets. Assume that $Y$ is $n$-spherical, and that for all $y\in Y,$ the poset $Y_{>y}$ is $(n-h(y)-1)$-spherical and the poset $f/y$ is $h(y)$-spherical. Then $X$ is $n$-spherical, and there is a descending filtration $F_\bullet$ of $\tilde{\HH}_n(X)$ such that $f_*$ factors through an isomorphism
		$$\tilde{\HH}_n(X)/F_0 \cong \tilde{\HH}_n(Y),$$ and where
		$$F_q/F_{q+1} \cong \bigoplus_{h(y)=q} \tilde{\HH}_{n-q-1}(Y_{>y})\tensor \tilde{\HH}_q(f/y).$$
	\end{thm}

	\subsection{The Partial Basis Complex}\;
	
	In this section, we discuss a complex which is closely related to the Tits complex.
	
	\begin{defn}
		The \emph{partial basis complex} $\cB_n(R)$ of a ring $R$ is the simplicial complex whose vertices are those vectors in $R^n$ which are contained in a basis of $R^n$, and where $\{v_1,\dots v_{k+1}\}$ are the vertices of a $k$-simplex if this set can be completed to a basis of $R^n$. 
	\end{defn} 
	The partial basis complex is not the nerve of a poset, but its barycentric subdivision is the nerve of the poset of partial bases of $R^n$, ordered by inclusion (see Example \ref{subdiv}). We call this poset $\bbB_n(R)$. 
	
	The following result may be deduced from the work of van der Kallen \cite{vdK}.
	\begin{thm}\label{vdK}
		Let $R$ be a ring satisfying $SR_k$ for $k\geq 2$. Then, the partial basis complex $\cB_n(R)$ is $(n-k)$-connected. It is Cohen-Macaulay if $k=2$.
	\end{thm}
	We give a precise exposition of this deduction. We need an auxiliary definition:
	\begin{defn}
		Let $\cU_n(R)$ be the simplicial complex where a set $\{v_1,\dots v_{k+1}\}$ of vectors in $R^n$ are the vertices of a $k$-simplex if this set spans a $(k+1)$-dimensional free direct summand of $R^n$.
	\end{defn}
	We have $\cB_n(R)\subseteq \cU_n(R)$. These complexes do not necessarily agree. A simplex $\sigma = \{v_1,\dots v_{k+1} \}$ of $\cU_n(R)$ lies in $\cB_n(R)$ exactly when the $\Span \{v_1,\dots v_{k+1} \}$ admits a free complement. For an example where this fails, see Example \ref{exmp:nongood}.

	\begin{thm}[van der Kallen, \cite{vdK}]\label{realvdK}
		Let $R$ be a ring satisfying $SR_k$ for $k\geq 2$. Then, the complex $\cU_n(R)$ is $(n-k)$-connected. It is Cohen-Macaulay if $k=2$.
	\end{thm}

	Here is how to deduce Theorem \ref{realvdK} from the results of \cite{vdK}. The first assertion is obtained by setting $\delta = 0$ in part (i) of \cite[Theorem from \S2.6]{vdK}. To deduce that $\cU_n(R)$ is Cohen-Macaulay when $k=2$, we apply Proposition \ref{prop:QuillenCM}. Let $\bbU_n(R)$ be the poset of simplices of $\cU_n(R)$. Let $\sigma$ be an $(\ell-1)$-simplex of $\cU_n(R$), let $V$ be the span of its vertices, and choose an isomorphism $V\cong R^\ell$. Then, $\bbU_n(R)_{<\sigma}$ is isomorphic to the $(\ell-2)$-skeleton of $\bbU_\ell(R)$, and hence is spherical by part (i) of \cite[Theorem from \S2.6]{vdK}. The connectivity of $\bbU_n(R)_{>\sigma}$ follows from part (ii) of \cite[Theorem from \S2.6]{vdK} (again setting $\delta = 0$). Now, we choose $\sigma'$ a partial basis containing $\sigma$, of size $\ell'$ and span $V'$, and an isomorphism $\phi:V'\cong R^{\ell'}$. Then, $(\bbU_n(R)_{>\sigma}\cap \bbU_n(R)_{<\sigma'})$ is isomorphic to the $(\ell'-\ell-2)$-skeleton of $\bbU_{\ell'}(R)_{>\phi(\sigma)}$, which we know to be spherical. Thus, Proposition \ref{prop:QuillenCM} applies, and $\cU_n(R)$ is Cohen-Macaulay.

	\begin{proof}[Proof of Theorem \ref{vdK}]
		We begin with the first assertion. It suffices to check that, if $R$ satisfies $SR_k$, then \[\cB_n(R)^{[n-k]} = \cU_n(R)^{[n-k]}.\] Once we have established this fact, by standard arguments we know that \[\pi_i(\cB_n(R)) = \pi_i(\cU_n(R)) = 0\text{ for }i\leq n-k-1.\] Likewise, $\HH_{n-k}(\cB_n(R)) = \HH_{n-k}(\cU_n(R))=0$, so by the Hurewicz theorem we know that $\pi_{n-k}(\cB_n(R)) = 0$, whence the first assertion. 
		
		We prove that $\cB_n(R)^{[n-k]} = \cU_n(R)^{[n-k]}$. Indeed, let $\sigma = \{v_1,\dots, v_{p+1}\}$ be a $p$-simplex, for $p\leq n-k$. Let $V_\sigma = \Span\{v_1,\dots, v_{p+1}\}$, and let $W$ be a complement to $V_\sigma$ in $R^n$. We observe that $W$ is a stably free $R$-module of rank $n-(p+1) \geq n-(n-k+1) = k-1$. By Theorem \ref{BassCancel}, $W$ is a free module, so that $\sigma$ lies in $\cB_n(R)$.
		
		For the second assertion, note that we have already proved that \[\cB_n(R)^{[n-2]} = \cU_n(R)^{[n-2]}\] when $R$ satisfies $SR_2$. It remains to check the $(n-1)$-simplices. But, an $(n-1)$-simplex corresponds to a basis of a rank $n$ direct summand of $R^n$. This must be a basis of $R^n$ itself. Therefore, \(\cB_n(R) = \cU_n(R)\). We already know that $\cU_n(R)$ is Cohen-Macaulay when $R$ satisfies $SR_2$, whence the theorem. 
	\end{proof}

	When $R$ is a Dedekind domain of arithmetic type, the main technical result of \cite{CFP} states that $\cB_n(R)$ is more connected than Theorem \ref{vdK} implies.

	\begin{thm}[Church-Farb-Putman, \cite{CFP}]\label{CFP}
		Let $\cO_S$ be a Dedekind domain of arithmetic type, where $S$ contains at least one non-complex place. Then, $\cB_n(\cO_S)$ is Cohen-Macaulay.
	\end{thm}
	
	\subsection{The Tits Complex}\;
	
	We discuss here the properties of the Tits complex $\cT_n(R)$. 
	\begin{defn}
		Let $R$ be a commutative ring. Denote by $\bbT_n(R)$ the poset of free and cofree summands of $R^n$, where $V\leq W$ when $V$ is a cofree summand of $W$. Note that this is stronger than $V\subseteq W$.
	\end{defn}
	The nerve of $\bbT_n(R)$ is $\cT_n(R)$. Therefore, by analyzing the poset $\bbT_n(R)$, we can understand $\cT_n(R)$.
	\begin{prop}
		Let $f:R\to S$ be a map of rings. Then, there is a simplicial map $\cT_n(R)\to \cT_n(S)$. This map is $\GL_n(R)$-equivariant. This assignment extends $\cT$ into a functor $\textbf{Ring}\to \textbf{Top}$.  
	\end{prop}
	\begin{proof}
		Consider the extension-of-scalars map 
		\begin{align*}
			\bbT_n(R)&\to\bbT_n(S)\\
			V&\mapsto V\tensor_R S.
		\end{align*}
		Extension of scalars takes free $R$-modules to free $S$-modules, and so sends a free and cofree summand of $R^n$ to such a summand of $S^n$. Therefore, this map is well-defined as a map of sets. Consider $V\leq W$ in $\bbT_n(R)$, so that $V/W$ is free. Since $V,\ W,$ and $V/W$ are all free and therefore flat, we have the exact sequence:
		$$0\to V\tensor_R S \to W\tensor_R S \to V/W \tensor_R S \to 0,$$
		It follows that $(V\tensor_R S)\leq (W\tensor_R S)$. Therefore, extension of scalars defines a map-of-posets $\bbT_n(R)\to \bbT_n(S)$. By functoriality of the nerve, we have the desired map $\cT_n(R)\to \cT_n(S)$.
		
		To see that this map is $\GL_n(R)$-equivariant, note that $\GL_n(R)$ acts on the poset $\bbT_n(R)$ in the natural way, as well as on $\bbT_n(S)$ via the map $\GL_n(R)\to \GL_n(S)$ induced by $f$. Extension of scalars is clearly compatible with these actions. Functoriality of the nerve then implies that the induced map on spaces is $\GL_n(R)$-equivariant.
	\end{proof}
	There are other reasonable generalizations of the Tits building to the category of rings -- say, by considering larger families of flags of free modules. Most of these will be identical when $R$ is local or a PID. We choose this particular definition for two reasons: first, the Tits complex is \emph{pure} for every ring, i.e., every simplex is contained in a simplex of the maximum dimension. Second, our Tits complex is closely connected to the partial basis complex, as we will see in Lemma \ref{span map}.
	
	We denote by $\cB'_n(R)$ the $(n-2)$-skeleton of $\cB_n(R)$, corresponding to those partial bases which are not complete. We denote by $\bbB'_n(R)$ the corresponding poset of incomplete partial bases. There is a map of posets $\bbB'_n(R)\to \bbT_n(R)$ defined by sending a partial basis of $R^n$ to the summand of $R^n$ it spans. Clearly this span is free and cofree. This induces a simplicial map on nerves, which we will denote 
	$$\Span\colon \sd\cB'_n(R)\to \cT_n(R).$$
	\begin{lem}\label{span map}
		The map $\Span\colon \sd \cB'_n(R)\to \cT_n(R)$ is surjective.
	\end{lem}
	\begin{proof}
		Let $\sigma = \{[V_1], \dots, [V_{k+1}]\}$ be a simplex of $\cT_n(R)$ corresponding to a flag $$0\subsetneq V_1 \subsetneq \dots \subsetneq V_{k+1} \subsetneq R^n.$$ Choose a basis of $V_1$. Since $V_1$ is cofree in $V_2$, this basis can be extended to a basis of $V_2$. Similarly, we can extend to a basis of $V_3$, and so on through $V_{k+1}$. This chain of partial bases defines a simplex of $\sd \cB'_n(R)$ that maps to $\sigma$. Since $\sigma$ was arbitrary, the map is surjective.
	\end{proof}

	\section{Connectivity of the Tits Complex}
	The goal of this section is to prove Theorem \ref{Connectivity}.
	We will prove this by induction, on both the degree of connectivity, and on the following natural filtration of $\cT_n(R)$:
	\begin{defn}
		Let $n>m\geq 1$, and let $R$ be a commutative ring. Define $\cT_{n,m}(R)$ to be the full subcomplex of $\cT_n(R)$ whose vertices consist of free and cofree summands of $R^n$ of rank at most $m$. 
	\end{defn}
	We will first prove the following statement:
	\begin{thm}\label{TitsFiltConn}
		Let $n>m\geq 1$, and let $R$ be a ring such that:
		\begin{enumerate}
			\item $\cB_n(R)$ is Cohen-Macaulay of dimension $n-2$.
			\item $\SL_k(R)$ is generated by elementary matrices for all $k\geq 1$.
			\item The natural map $K_2(3,R)\to K_2(k,R)$ is surjective for all $k\geq 3$. 
		\end{enumerate} 
		Then, $\cT_{n,m}(R)$ is $(m-2)$-connected.
	\end{thm}
	The proof of this result comprises \S4.1--4.3. Our strategy is induction on the degree of connectivity. Our first base case is $m=1$.
	\begin{rem}\label{rem:nonempty}
		The space $\cT_{n,1}(R)$ is nonempty, since $[Re_1]\in\cT_{n,1}$. Therefore, $\cT_{n,1}(R)$ is $(-1)$-connected.
	\end{rem}
	Our second base case is to show that $\cT_{n,m}(R)$ is connected when $m\geq 2$. This is true quite generally, and is proved in that generality in \S4.1. Our third base case, treated in \S4.2, is simple-connectivity for $m\geq 3$. This is more delicate, and is where we require conditions (2) and (3). In \S4.3 we compute the homology of $\cT_{n,m}(R)$ using (1), showing that it is concentrated in the top dimension. 
	
	Finally, in \S4.4, we give a proof of Theorem \ref{Connectivity}. This is an immediate consequence of Theorem \ref{TitsFiltConn} in all but a few cases.
	
	\subsection{Path-connectedness of $\cT_{n,m}(R)$}\;
	
	In this section, we prove that $\cT_{n,m}(R)$ is connected when $m\geq 2$. We do this twice, with the second proof giving additional information needed in \S4.2. Our first proof is an immediate deduction from the connectivity of $\cB_n(R)$:
	\begin{thm}\label{0connv2}
		Let $n>m\geq 2$, and let $R$ be a ring satisfying $SR_n$. Then, $\cT_{n,m}(R)$ is path-connected.
	\end{thm}
	\begin{proof}
		By van der Kallen's theorem (Theorem \ref{vdK}), we know that $\cB_n(R)$ is path-connected. Therefore, since $m\geq 2$, so is its $(m-1)$-skeleton $\cB_n(R)^{[m-1]}$. By Lemma \ref{span map}, the map $$\Span: \cB_n(R)^{[m-1]}\to \cT_{n,m}(R)$$ defined in \S3.2 is surjective and continuous. Therefore, $\cT_{n,m}(R)$ is path-connected.
	\end{proof}
	
	We now give a second proof, which will additionally furnish us with a map from the Cayley graph of $\SL_n(R)$. We will need this map in \S\ref{section:1Conn} to prove simple-connectedness.
	
	To begin, we recall the definition of the Cayley graph:
	\begin{defn}
		Let $G$ be a group, and let $S$ be a set of generators for $G$. The \emph{Cayley graph} $\Gamma(G,S)$ is the graph whose vertices are the elements of $G$, and with an edge $\{g,g'\}$ if $g' = gs$ for some $s\in S$. The space $\Gamma(G,S)$ has a natural action of $G$ by graph isomorphisms, induced by left multiplication. 
	\end{defn}
	
	Since $S$ generates $G$, $\Gamma(G,S)$ is always connected. 
	
	The obvious choice of a generating set for $\SL_n(R)$ is the set of elementary matrices. We choose a slightly larger set, for the purposes of \S4.2. For $r\in R^\times$, define $$t(r) = \begin{bmatrix} r & & & \\ & r^{-1} & & \\ & & 1 & \\  & & & \ddots \end{bmatrix}\in \SL_n(R).$$
	This is a product of elementary matrices, using the identity \begin{align*}
		\begin{bmatrix}
			r & 0 \\
			0 & r^{-1} \\
		\end{bmatrix} = 
	\begin{bmatrix}
		1 & r-1 \\
		0 & 1 \\
	\end{bmatrix}
	.
	\begin{bmatrix}
		1 & 0 \\
		1 & 1 \\
	\end{bmatrix}
	.
	\begin{bmatrix}
		1 & r^{-1}-1 \\
		0 & 1 \\
	\end{bmatrix}
	.
	\begin{bmatrix}
		1 & 0 \\
		-r & 1 \\
	\end{bmatrix}.
	\end{align*}
	
	We define the generating set:
	$$\cS = \{e_{ij}(\lambda)\mid i\ne j,\ \lambda\in R\} \cup \{t(r) \mid r \in R^\times\}.$$
	Adding these extra generators only adds the relations $$t(r) = e_{12}(r-1)e_{21}(1)e_{12}(r^{-1}-1)e_{21}(-r)$$ to the presentation of $\SL_n(R)$ found in \S\ref{Ktheory}.
	
	\begin{thm}\label{0conn}
		Let $n>m\geq 2$. For a ring $R$ such that $\SL_n(R)$ is generated by elementary matrices, the following hold:
		\begin{itemize}
			\item The complex $\cT_{n,m}(R)$ is path-connected.
			\item There is a continuous $\SL_n(R)$-equivariant map $$\phi\colon \Gamma\left(\SL_n(R),\cS\right)\to \cT_{n,m}(R)$$ which sends the vertex set of $\Gamma\left(\SL_n(R),\cS\right)$ surjectively onto $\cT_{n,1}(R)$.
		\end{itemize} 
	\end{thm}
	\begin{proof}
		If $V$ is a free and cofree summand of $R^n$, it contains a 1-dimensional free and cofree summand $L$, and so there is an edge from $[V]$ to $[L]$. By this fact, we will have proved the first assertion if we can show that all 1-dimensional summands lie in a single path component. Since Cayley graphs are path-connected, this is an immediate consequence of the second assertion. 
		
		We now prove the second assertion. We explicitly construct the desired map $\phi$. We first define it on vertices, sending the vertex $\gamma$ to $\gamma\cdot Re_1$. Since $\SL_n(R)$ acts transitively on $\cT_{n,1}(R)$ by Corollary \ref{flagTrans}, this is surjective onto $\cT_{n,1}(R)$.
		
		Notice that every edge of the Cayley graph is of the form $\{\gamma, \gamma s\},$ for some $\gamma\in \SL_n(R)$ and $s\in\cS$. Using the natural action of $\SL_n(R)$, the edge $\{\gamma, \gamma s\}$ is equal to $\gamma\cdot \{\id,s\}$.  It therefore suffices to define the image of $\{\id,s\}$ for each generator $s\in S$. There will then be a unique $\SL_n(R)$-equivariant extension over all of $\Gamma(\SL_n(R),\cS)$. 
		
		Defining the image of $\{\id ,s\}$ amounts to choosing a path in $\cT_{n,m}(R)$ from $Re_1$ to $s\cdot Re_1$. If $s = e_{i,j}(\lambda)$ with $j\ne 1$, or if $s=t(r)$, then $s\cdot Re_1 = Re_1$, and we choose the constant path. If $s = e_{i,1}(\lambda)$, then $s\cdot Re_1 = R(e_1+\lambda e_i)$ and we choose the path: \[\xymatrix{ [Re_1] \ar@{-}[r] &[Re_1\oplus Re_i] \ar@{-}[r] & [R(e_1+\lambda e_i)] }\qedhere\]
	\end{proof}
	
	\subsection{Simple Connectedness of $\cT_{n,m}(R)$}\label{section:1Conn}\;
	
	Our goal in this section is to show that $\cT_{n,m}(R)$ is simply connected for $m\geq 3$, using the approach of \cite{PutmanTrick}. 
	
	\begin{thm}\label{1conn}
		Let $n>m\geq 3$. Let $R$ be a ring such that $\SL_k(R)$ is generated by elementary matrices for all $k\geq 1$, and the natural map $K_2(3,R)\to K_2(n,R)$ is surjective. Then, the complex $\cT_{n,m}(R)$ is simply connected.
	\end{thm}
	
	We will prove Theorem \ref{1conn} at the end of the section. We begin with some preliminary results.
	
	\begin{lem}\label{lem:loop1skel}
		Let $R$ either satisfy $SR_3$, or be such that $\SL_k(R)$ is generated by elementary matrices for all $k$. Then, every loop based at $[Re_1]$ in $\cT_{n,m}(R)$ is homotopic rel basepoints to a loop in $\cT_{n,2}(R)$. 
	\end{lem}
	\begin{proof}
		Let $\gamma$ be such a loop. Without loss of generality, it is a simplicial loop in the 1-skeleton of $\cT_{n,m}(R)$. We will show that if $m\geq 3$, then $\gamma$ may be homotoped into the 1-skeleton of $\cT_{n,m-1}(R)$. Applying this repeatedly proves the result. 
		
		Let $[V]$ be a $m$-dimensional summand appearing in $\gamma$, and let $W_1$ and $W_2$ be the summands encountered immediately before and after $V$. Note that the ranks of $W_1$ and $W_2$ must be strictly smaller than the rank of $V$, since $V$ has maximal rank in $\cT_{n,m}(R)$, and each edge corresponds to the inclusion of a direct summand. 
		
		Since $V$ is $m$-dimensional, the link of $[V]$ in $\cT_{n,m}(R)$ is the realization of the poset of free and cofree direct summands of $R^n$ contained in $V$. This poset is isomorphic to the poset of direct summands of $V$, so $\lk [V] \cong \cT_{m}(R).$ Since $m\geq 3$ and $R$ satisfies the hypotheses of at least one of Theorem \ref{0connv2} and Theorem \ref{0conn}, we know that $\lk [V]$ is connected. Choose a simplicial path in $\lk [V]$ from $[W_1]$ to $[W_2]$. Since the star $\str[V]$ is contractible, the segment $[W_1] - [V] - [W_2]$ is homotopic to this path. This extends to a homotopy from $\gamma$ to a loop that misses $[V]$. Since $\lk [V]$ is a subset of $\cT_{n,m-1}(R)$, this homotopy reduces the number of $m$-dimensional summands in $\gamma$ by 1. Repeating this process, we may homotope $\gamma$ into $\cT_{n,m-1}(R)$, as desired.
	\end{proof}
	We need a fact about the stabilizer of $Re_1$:
	\begin{lem}\label{lem:stab}
		Let $R$ be such that $\SL_k(R)$ is generated by elementary matrices for all $k$. Then, the $\SL_n(R)$-stabilizer of $Re_1$ is generated by the subset of $\cS$ consisting of elements stabilizing $Re_1$:  
		$$\{e_{ij}(\lambda)\mid i\ne j,\ \lambda\in R,\ j\ne 1 \}\cup \{t(r)\mid r\in R^\times\}$$
	\end{lem}
	\begin{proof}
		Fix an arbitrary matrix in the the stabilizer of $Re_1$. This matrix has the form
		$\left[\begin{array}{c|c} u & \vec{v} \\ \hline 0 & A \end{array}\right]$, for $u\in R^\times$, $\vec{v}\in R^{n-1}$, and $A\in \GL_{n-1}(R)$. We have $u=\det(A)^{-1}$. We will find it as a product of generators stabilizing $Re_1$. 
		
		Set $\vec{w} = u^{-1}\vec{v}$, and note that the matrix $\left[\begin{array}{c|c} 1 & \vec{w} \\ \hline 0 & \id \end{array}\right]$ is clearly a product of elementary matrices stabilizing $Re_1$, and that $\left[\begin{array}{c|c} u & \vec{v} \\ \hline 0 & A \end{array}\right] = \left[\begin{array}{c|c} u & 0 \\ \hline 0 & A \end{array}\right]\left[\begin{array}{c|c} 1 & \vec{w} \\ \hline 0 & \id \end{array}\right]$.
		
		Set $$A' = A \begin{bmatrix} u & & \\ & 1 & \\ & & \ddots \end{bmatrix}.$$ Since $A'\in \SL_{n-1}(R)$ and $\SL_{n-1}(R)$ is generated by elementary matrices, $\left[\begin{array}{c|c} 1 & 0 \\ \hline 0 & A' \end{array}\right]$ is a product of elementary matrices $e_{i,j}(\lambda)$ with $i,j\geq 2$, each of which stabilize $Re_1$. Since $$ \left[\begin{array}{c|c} u & \vec{v} \\ \hline 0 & A \end{array}\right] = \left[\begin{array}{c|c} 1 & 0 \\ \hline 0 & A' \end{array}\right]t(u) \left[\begin{array}{c|c} 1 & \vec{v} \\ \hline 0 & \id \end{array}\right]$$
		and $t(u)$ stabilizes $Re_1$, we have proved the statement.
	\end{proof}

	The hypotheses of Theorem \ref{1conn} allow us to apply Theorem \ref{0conn}, giving us the equivariant map $\phi\colon\Gamma(\SL_n(R), \cS)\to\cT_{n,m}(R)$ constructed by that Theorem. In the next lemma, we will show that any loop in $\cT_{n,m}(R)$ is the image of a loop in $\Gamma(\SL_n(R), \cS)$.

	\begin{lem}\label{lem:preprocessLoops}
		Let $R$ be such that $\SL_k(R)$ is generated by elementary matrices for all $k$. Then, every loop based at $[Re_1]$ in $\cT_{n,m}(R)$ is homotopic to the image under $\phi$ of a loop based at $\id$ in $\Gamma(\SL_n(R), \cS)$.
	\end{lem}
	\begin{proof}
		By Lemma \ref{lem:loop1skel}, it suffices to consider oriented loops $\gamma$ in $\cT_{n,2}(R)$. Such a loop encounters vertices corresponding to 1-dimensional and 2-dimensional summands, in an alternating manner. It is thus composed of segments of the following form:
		$$\xymatrix{ [L_1] \ar@{-}[r] &[P] \ar@{-}[r] & [L_2] }$$
		where the $L_i$ are 1-dimensional, and $P$ is 2-dimensional. By the construction in Lemma \ref{0conn}, we may lift this to a segment in $\Gamma(\SL_n(R), \cS)$ if there is some $g\in \SL_n(R)$ such that, for some choice of $i$ and $\lambda$, we have 
		\begin{align*}
			L_1 &= g\cdot Re_1 \\
			L_2 &= g\cdot R(e_1+\lambda e_i)\\
			P &= g\cdot (Re_1\oplus Re_i).
		\end{align*}  We call such segments ``liftable".
		
		We will find a homotopy of $\gamma$ to a loop where each segment of the above form is liftable. Fix such a segment in $\gamma$. By Corollary \ref{flagTrans} applied to flags of type $(1,1,n-2)$, we may choose $g\in \SL_n(R)$ such that $g\cdot Re_1 = L_1$ and $g\cdot (Re_1 \oplus Re_2) = P$. Write $g^{-1}L_2 = L_2'$, which is a free summand of $Re_1\oplus Re_2$. Since $\SL_2(R)$ acts transitively on $\cT_2(R)$, we may choose $g'\in \SL_2(R)$ such that $g'\cdot Re_1 = L_2'$. Express $g'$ as a product of elementary matrices $s_1\dots s_k$ in $\SL_2(R)$ (which are also generators for $\SL_n(R)$). Consider the path $\rho$:
		\begin{align*}
			[Re_1] - [Re_1\oplus Re_2] -  [s_1\cdot Re_1] - [Re_1\oplus Re_2&] -  [s_1s_2\cdot Re_1] -  \dots\\
			&\dots - [Re_1\oplus Re_2] -  [g'\cdot Re_1] = L_2'.
		\end{align*}
		
		Each segment of this path is liftable, by construction (notice that $s_i\cdot (Re_1\oplus Re_2) = Re_1\oplus Re_2$). The path $g\cdot \rho$ is also homotopic to the original segment, since each edge traversed (aside from the first and last) is immediately doubled back upon. Therefore we may homotope $\gamma$ to replace the original segment with this new path. This decreases by 1 the number of nonliftable segments in $\gamma$. Repeating this process, we can homotope $\gamma$ to a path entirely composed of liftable segments.
		
		Now, we may lift $\gamma$, segment by segment, to a path $p$ in $\Gamma(\SL_n(R),\cS)$ beginning at $\id$. The end of the path is not necessarily $\id$, but since the image under $\phi$ of this endpoint is $[Re_1]$, it must be in the stabilizer of $[Re_1]$. 
		
		By Lemma \ref{lem:stab}, the stabilizer of $[Re_1]$ is generated by a subset of $\cS$, and so defines a connected subgraph of $\Gamma(\SL_n(R),\cS)$. Any path contained in this subgraph will map to the constant path in $\cT_{n,m}(R)$. Choosing some such path $p'$, we define the loop $\tilde{\gamma}$ to be the composition $p'\ast p$. We have $\phi(\tilde{\gamma})\simeq\gamma$ (where the homotopy is a reparametrization of the domain).
	\end{proof}

	We are now in a position to prove Theorem \ref{1conn}.
	
	\begin{proof}[Proof of Theorem \ref{1conn}]
		By Lemma \ref{lem:preprocessLoops}, the map $\phi\colon \Gamma(\SL_n(R),\cS)\to \cT_{n,m}(R)$ is surjective on $\pi_1$. So, it suffices to show that the image of a based loop in $\Gamma(\SL_n(R),\cS)$ bounds a disk in $\cT_{n,3}(R)$.
		Every such loop corresponds to some word in the generators representing the identity. Such a word is a product of conjugates of the defining relators of a given presentation of $\SL_n(R)$. 
		
		Since the map $\phi$ is $\SL_n(R)$-equivariant, if we show that the image of a loop in $\Gamma(\SL_n(R),\cS)$ corresponding to a defining relator bounds a disk, we will prove the same for the image of any loop. This shows that the map on $\pi_1$ induced by $\phi$ is equal to zero, whence the Theorem.
		
		Combining Theorem \ref{thm:SLnPres} and the remark immediately preceding Theorem \ref{0conn}, there are 3 kinds of defining relators in our presentation of $\SL_n(R)$. These are the Steinberg relations appearing in $S_n(R)$, the relations corresponding to generators of $K_2(3,R)$, and the relations $t(r)= e_{12}(r-1) e_{21}(1)e_{12}(r^{-1}-1)e_{21}(-r)$. We list some cases where defining relators will correspond to trivial loops in $\pi_1(\cT_{n,3}):$
		\begin{itemize}
			\item If a relation is a product of elementary matrices $e_{i_aj_a}(\lambda)$ where $j_a$ is never $1$. Then, each of these generators stabilize $Re_1$, so the associated loop in $\Gamma(\SL_n(R),\cS)$ maps to the constant loop in $\cT_{n,m}(R)$.
			\item If a relation is a product of elementary matrices whose indices are drawn from 3 distinct values $\{1,i,j\}$, then each generator is an automorphism of $Re_1\oplus Re_i\oplus Re_j$. Since the flags $Re_1\subsetneq Re_1\oplus Re_i$ and $R(e_1+\lambda e_i)\subsetneq Re_1\oplus Re_i$ can both be extended by $Re_1\oplus Re_i\oplus Re_j$, it follows that the flag corresponding to any edge in the loop can be extended by $Re_1\oplus Re_i\oplus Re_j$. This shows that this loop is contained in the star of $[Re_1\oplus Re_i\oplus Re_j]$, and so bounds a disk.
		\end{itemize}
				
		These observations handle nearly all of the defining relations. The relations corresponding to generators of $K_2(3,R)$ are products of elementary matrices $e_{ij}(\lambda)$ with $i,j\in \{1,2,3\}$. A Steinberg relation will satisfy one of the above conditions, unless it is of the form $[e_{ij}(\lambda),e_{k\ell}(\mu)]$ with $i\ne \ell,\ j\ne k$, and where exactly one of $j,\ell$ is equal to 1. The loop in $\Gamma(\SL_n(R),\cS)$ corresponding to such a relation consists of four segments. The image of two of these segments will be constant, and of the remaining two, one will be the reverse of the other. Therefore, the image of the loop corresponding to any Steinberg relation is trivial.
		
		The relation $t(r^{-1}) e_{12}(r-1) e_{21}(1)e_{12}(r^{-1}-1)e_{21}(-r)$ is not a product of elementary matrices, so the above observations do not apply. Notice that every generator is an automorphism of $Re_1\oplus Re_2$. Therefore, the image of the corresponding loop lies in the star of $[Re_1\oplus Re_2]$, and therefore bounds a disk.
		
		We have now checked the theorem for all defining relators, as required.		
	\end{proof}
	
	\subsection{Homological connectivity of $\cT_{n,m}(R)$}\;
	
	The goal of this section is a proof of the following theorem:
	
	\begin{thm}\label{HomologyConn}
		Let $n>m> 0$, and $R$ be a ring such that $\cB_k(R)$ is Cohen-Macaulay for all $k$. Assume that $\cT_{k,\ell}(R)$ is connected for $\ell\geq 2$ and simply connected for $\ell\geq 3$. Then, $\cT_{n,m}(R)$ is $(m-2)$-connected.
	\end{thm}
	\begin{rem}
		If one drops the assumptions on connectivity and simple-connectivity, the arguments given in this section will prove the weaker statement that $\cT_{n,m}(R)$ is $(m-2)$-acyclic.
	\end{rem}
	
	We will prove Theorem \ref{HomologyConn} by induction on $n$ and $m$. The structure of this induction is as follows. Consider the set of pairs of natural numbers $(n,m)$ satisfying $n>m>0$, under the dictionary ordering. This is a linear order, isomorphic to $\bbN$. Remark \ref{rem:nonempty} implies the theorem is true when $m=1$. By assumption, the theorem is true when $m=2,3$. These are our base cases. The statement we must prove is the following:
	\begin{IndProp}
		Fix $n>m\geq 3$. Assume $\cT_{n',m'}(R)$ is $(m'-2)$-connected for any $(n',m')\leq (n,m)$ in the dictionary ordering. Then, $\cT_{n,m+1}(R)$ is $(m-1)$-connected.
	\end{IndProp}
	The remainder of this section is devoted to a proof of the Inductive Proposition.
	First, we want to better understand the ``filtration quotients" of $\cT_{n,\bullet}(R)$. Recall that for a ring $R$, the Grassmannian $\Gr_k^n(R)$ is the set of rank $k$ free and cofree summands of $R^n$.
	\begin{prop}\label{FiltQuot}
		The cofiber of the inclusion $i_\cT\colon\cT_{n,m}(R)\hookrightarrow \cT_{n,m+1}(R)$ is homotopy equivalent to $$\bigvee_{\mathclap{\Gr_{m+1}^n(R)}}\; \Sigma\; \cT_{m+1}(R).$$
	\end{prop}
	\begin{proof}
		Since $\cT_{n,m}(R)$ is a closed subcomplex of $\cT_{n,m+1}(R)$, the cofiber of the inclusion is homotopy equivalent to the quotient space $\cT_{n,m+1}(R)/\cT_{n,m}(R)$. We will show this is homeomorphic to the indicated space.
		
		Every vertex $[V]$ of $\cT_{n,m+1}(R)\setminus \cT_{n,m}(R)$ corresponds to a rank $m+1$ free summand $V$ of $R^n$. Every vertex of $\cT_{n,m+1}(R)$ connected to $[V]$ must correspond to a free and cofree summand of $V$. Therefore, we see that $\lk [V] \subseteq \cT_{n,m}(R)$, and $\lk[V]\cong \cT_{m+1}(R$). Therefore, $\str [V] / \lk [V]$ is homeomorphic to $\Sigma \lk [V]$. These assemble into a map $$f:\bigvee_{\mathclap{V\in \Gr_{m+1}^n(R)}}\; \Sigma \cT_{m+1}(R) \to \cT_{n,m+1}(R)/\cT_{n,m}(R).$$
		
		There is a continuous partial inverse map defined on the image of the interior of $\str [V]$ in $\cT_{n,m+1}(R)/\cT_{n,m}(R)$. Since these sets do not overlap, these partial maps may be glued together into a continuous inverse to $f$, sending the point corresponding to $\cT_{n,m}(R)$ to the wedge point of $\displaystyle \bigvee_{\mathclap{\Gr_{m+1}^n(R)}}\; \Sigma \cT_{m+1}(R)$.
	\end{proof}
	The following result is deduced from the cofiber sequence in reduced homology:
	\begin{cor}
		If $\cT_{n,m}(R)$ and $\cT_{m+1}(R)$ are $(m-2)$-connected, then \[\tilde{\HH}_q(\cT_{n,m+1}(R)) =0\ \text{ for } q\leq m-2,\] and there is an exact sequence: 
		\begin{align*}
			0 \to \tilde{\HH}_m(\cT_{n,m+1}(R))\to \bigoplus\tilde{\HH}_{m-1}&(\cT_{m+1}(R))\to\\ 
			&\to \tilde{\HH}_{m-1}(\cT_{n,m}(R))\to \tilde{\HH}_{m-1}(\cT_{n,m+1}(R)) \to 0
		\end{align*}
	\end{cor}
	
	This gets us much of the way to our goal. To prove Theorem \ref{HomologyConn}, it suffices to show that the map $$(i_\cT)_*\colon\tilde{\HH}_{m-1}(\cT_{n,m}(R))\to \tilde{\HH}_{m-1}(\cT_{n,m+1}(R))$$ is zero. To do this, we compare with $\cB_n(R)$. Denote by $\cB_n(R)^{[m]}$ the $m$-skeleton of the partial basis complex -- this consists of partial bases of size $m+1$, and so $\Span(\sd \cB_n(R)^{[m]}) = \cT_{n,m+1}(R)$. We have the following commutative diagram of spaces:
	$$\xymatrix{\cB_n(R)^{[m-1]} \ar@{^{(}->}[r]^{i_\cB} \ar[d]^{\Span} & \cB_n(R)^{[m]} \ar[d]^{\Span} \\ \cT_{n,m}(R) \ar@{^{(}->}[r]^{i_\cT} & \cT_{n,m+1}(R)}$$
	where $i_\cB$ is the inclusion of the $(m-1)$-skeleton into the $m$-skeleton. Taking homology, we have 
	$$\xymatrix{\tilde{\HH}_{m-1}(\cB_n(R)^{[m-1]}) \ar[r]^{0} \ar[d]^{\Span_*} & \tilde{\HH}_{m-1}(\cB_n(R)^{[m]}) \ar[d]^{\Span_*} \\ \tilde{\HH}_{m-1}(\cT_{n,m}(R)) \ar[r]^{(i_\cT)_*} & \tilde{\HH}_{m-1}(\cT_{n,m+1}(R)) }$$
	where we know $(i_\cB)_* = 0$ by van der Kallen's calculation of the connectivity of $\cB_n(R)$ (see Theorem \ref{vdK}). It will suffice to prove that the leftmost span map is surjective on $\tilde{\HH}_{m-1}$. A diagram chase then shows that $(i_\cT)_* = 0$. This fact will also be of interest later:
	\begin{lem}\label{AptsSurj}
		Let $R$ be such that $\cB_n(R)$ is Cohen-Macaulay, and assume that $\cT_{n',m'}(R)$ is $(m'-1)$-spherical for any $(n',m')\leq (n,m)$. Then, $$\Span_*\colon\tilde{\HH}_{m-1}(\cB_n(R)^{[m-1]})\to \tilde{\HH}_{m-1}(\cT_{n,m}(R))$$ is surjective.
	\end{lem}
	\begin{proof}
		For this proof, we use the language of \S\ref{simptools}. Our goal is to invoke Theorem \ref{Quillen}. We check the conditions:
		\begin{itemize}
			\item \emph{For all $y\in \cT_{n,m}(R)$, $\Span/y$ is $h(y)$-spherical}. Indeed, if $y = [V]$ where $V$ is a rank $k$ free and cofree summand of $R^n$, then $h(y) = k-1$, and $\Span/y$ is the complex of partial bases of $V$. This is homeomorphic to $\cB_k(R)$, which is $(k-1)$-spherical by Theorem \ref{vdK}.
			\item \emph{For all $y\in \cT_{n,m}(R)$, $(\cT_{n,m}(R))_{>y}$ is $(m-h(y)-2)$-spherical}. As above, let $y = [V]$ where $V$ is a rank $k$ free and cofree summand of $R^n$, so that $h(y) = k-1$. In this case, $(\cT_{n,m}(R))_{>y}$ is the complex of free and cofree summands of $R^n$, of rank no more than $m$, containing $V$ as a cofree summand. This is isomorphic to the poset of free and cofree summands of $R^n/V \cong R^{n-k}$ of rank no more than $(m-k)$. Therefore, $(\cT_{n,m}(R))_{>y}\cong \cT_{n-k,m-k}(R)$, which is $(m-k-1)$-spherical by the inductive hypothesis.
		\end{itemize}
		By Theorem \ref{Quillen}, $\Span_*$ is surjective.
	\end{proof}

	\begin{proof}[Proof of Theorem \ref{HomologyConn}]
		The above arguments have shown that, if we assume that $\cT_{n',m'}(R)$ is $(m'-2)$-connected for any $(n',m')\leq (n,m)$ in the dictionary ordering, $\cT_{n,m+1}(R)$ is $(m-1)$-acyclic. Since $m\geq 3$, by hypothesis $\cT_{n,m+1}(R)$ is simply connected, so by the Hurewicz theorem, it is in fact $(m-1)$-connected. This proves the Inductive Proposition. Since the base case holds, either by hypothesis or by Remark \ref{rem:nonempty}, the theorem follows.
	\end{proof}
	
	 We are now in a position to prove Theorem \ref{TitsFiltConn}, which states that $\cT_{n,m}(R)$ is $(m-1)$-spherical if $R$ is a ring where:
	\begin{enumerate}
		\item $\cB_n(R)$ is Cohen-Macaulay.
		\item $\SL_k(R)$ is generated by elementary matrices for all $k\geq 1$.
		\item If $k\geq 3$, the natural map $K_2(3,R)\to K_2(k,R)$ is surjective.
	\end{enumerate}
	
	\begin{proof}[Proof of Theorem \ref{TitsFiltConn}]
		If $R$ satisfies the hypotheses of the Theorem, Theorems \ref{0conn} and \ref{1conn} both apply. This means that the hypotheses of Theorem \ref{HomologyConn} are satisfied, proving the Theorem.
	\end{proof}
	
	\subsection{Proof of Theorem \for{toc}{\ref*{Connectivity}}\except{toc}{\ref{Connectivity}}}\;
	
	Recall that Theorem \ref{Connectivity} asserts that $\cT_n(R)$ is Cohen-Macaulay when $R$ satisfies either $SR_2$ or is a Dedekind domain of arithmetic type. Since $\cT_{n,n-1}(R) = \cT_n(R)$, whenever Theorem \ref{TitsFiltConn} applies, we know that $\cT_n(R)$ is $(n-2)$-spherical. By the results of \S2, this includes all rings satisfying $SR_2$, and any Dedekind domain of arithmetic type $\cO_S$ which is either Euclidean, or where $S$ contains at least 2 places, at least one of which is not complex.
	
	To handle the case of a general Dedekind domain of arithmetic type, we want to understand the impact of localization on $\cT_n$. We recall the Picard group:
	\begin{defn}
		If $R$ is a commutative ring, then its \emph{Picard group} $\Pic(R)$ is the set of rank 1 projective modules, made into an abelian group using $\tensor_R$.
	\end{defn}
	The following two results are classical.

	\begin{prop}
		Let $\tilde{K}_0(R)$ be the set of projective $R$-modules modulo the relation $P\sim Q$ if $$P\oplus R^k \cong Q\oplus R^\ell \text{ for some }k,\ell \geq 0,$$ made into an abelian group under $\oplus$. If $R$ is a Dedekind domain, then $\Pic(R)$ is canonically isomorphic to $\tilde{K}_0(R)$.
	\end{prop}
	This is \cite[Corollary II.2.6.3]{KBook}, combined with the observation that $\Spec(R)$ is connected when $R$ is a domain. Notice that if $R$ is a PID, then $\Pic(R)$ is trivial.
	\begin{prop}\label{lem:picLocalize}
		Let $R$ be a Dedekind domain, $S$ a multiplicatively closed subset, and $R_S = S^{-1}R$. Identify $\Spec(R_S)$ with its image in $\Spec(R)$, and assume it is cofinite. Let $D(R,R_S)$ be the free abelian group on the prime ideals $\fp$ of $R$ where $\fp\cap S \ne \emptyset$. Then, we have an exact sequence:
		$$D(R,R_S)\to \Pic(R)\to \Pic(R_S)\to 0,$$
		where the map $D(R,R_S)\to \Pic(R)$ sends the generator corresponding to a prime $\fp$ to $[\fp]\in\Pic(R)$.
	\end{prop}
	This follows from \cite[Proposition II.6.5]{Hartshorne}, in the following way. Set $X=\Spec(R)$ and $U = \Spec(R_S)$, and interpret the Picard group as the ideal class group of fractional ideals. Surjectivity of the map $\Pic(R)\to \Pic(R_S)$ is part (a) of the cited proposition. As in the proof of part (c), we notice that the kernel of this map consists of Weil divisors whose support lies in $\Spec(R)\setminus\Spec(R_S)$. Define the map $D(R,R_S)\to \Pic(R)$ to send a prime ideal $\fp$ to its image in $\Pic(R)$. This surjects onto the kernel, proving the proposition.
	\begin{defn}
		Let $R$ be a Dedekind domain, and $H$ a subgroup of $\Pic(R)$. Define $\cT_n(R,H)$ to be the nerve of the poset $\bbT_n(R,H)$ of direct summands $P$ of $R^n$, such that $[P]$, thought of as an element of $\tilde{K}_0(R)$, lies in $H$. Call this the \emph{$H$-controlled Tits complex}.
	\end{defn}
	\begin{rem}
		Since $R$ is a Dedekind domain, a projective module is stably free if and only if it is free, and a free summand of a free module is automatically cofree. Therefore, $\cT_n(R) = \cT_n(R,\{1\})$.
	\end{rem}
	\begin{thm}\label{thm:localTits}
		Let $R$ be a Dedekind domain, and let $S$ be a multiplicatively closed set. Write $R_S = S^{-1}R$, and let $H=\ker\left(\Pic(R)\to \Pic(R_S)\right).$
		Then, $$\cT_n(R_S)\cong \cT_n(R,H)$$
	\end{thm}
	\begin{proof}
		We prove that these are nerves of isomorphic posets. Let $P$ be a projective module of rank $k$ over $R$, such that $[P]\in H$. Our definition of $H$ tells us that $P\otimes_{R} R_S$ is a free module of rank $k$ over $R_S$. If $P\leq Q$ in $\bbT_n(R,H)$, then the complement $Q/P$ satisfies $[Q/P] = [Q]-[P] \in H$. Since projectives are flat, all derived functors vanish and $$0\to P\otimes_{R} R_S \to Q\otimes_{R} R_S \to (Q/P)\otimes_{R} R_S \to 0$$is an exact sequence of free modules. In particular, $(P\otimes_{R} R_S)\leq (Q\otimes_{R} R_S)$. Therefore, extension of scalars produces a map $\bbT_n(R,H)\to \bbT_n(R_S)$.
		
		To construct an inverse, notice that $R^n$ is a subset of $R_S^n$. If $V$ is a rank $k$ free summand of $R_S^n$, consider $(V\cap R^n)\subseteq R^n$. This is finitely generated and torsion-free. Since $R$ is a Dedekind domain, $(V\cap R^n)$ is therefore projective. The quotient $R^n/(V\cap R^n)$ is a submodule of $R_S/V\cong R_S^{n-k}$, and so is also torsion-free and finitely generated. Therefore, the quotient map splits, so $(V\cap R^n)$ is in fact a direct summand of $R^n$. 
		
		To show that $(V\cap R^n)\tensor_{R}R_S = V$, notice that $(V\cap R^n)$ is projective of rank $k$, so $(V\cap R^n)\tensor_{R}R_S$ is a full rank submodule of $V$, and therefore is isomorphic to $V$. This additionally shows that $[V\cap R^n]\in H$. Therefore, $(-\cap R^n)$ is an inverse to $(-\tensor_R R_S)$, and we have constructed an isomorphism $\bbT_n(R,H)\cong \bbT_n(R_S)$.
	\end{proof}
	
	\begin{cor}
		Let $R$ be a Dedekind domain. Then, $\cT_n(R,\Pic(R))\cong \cT_n(\text{Frac } R)$. In particular, if $R$ is a PID, then $\cT_n(R)\cong \cT_n(\text{Frac } R)$
	\end{cor}
	\begin{proof}
		Let $S = R\setminus\{0\}$, so that $S^{-1}R = \text{Frac } R$. In this case, $\Pic(\text{Frac } R) = 1$, so $\ker\left(\Pic(R)\to \Pic(\text{Frac } R)\right)= \Pic(R)$.
	\end{proof}
	In some cases, localization does not change the Tits complex. This is crucial:
	\begin{lem}\label{CheekyLocalize}
		Let $\cO$ be a Dedekind domain, and let $s\in \cO$ generate a principal prime ideal. Then, $$\cT_n(\cO[\tfrac{1}{s}])\cong \cT_n(\cO).$$
	\end{lem}
	\begin{proof}
		By Theorem \ref{thm:localTits}, it suffices to prove that $\ker\left(\Pic(\cO)\to \Pic( \cO[\frac{1}{s}])\right)$ is trivial.		
		Let $L\subseteq\cO^n$ be a rank 1 direct summand such that $L[\frac{1}{s}]$ is free, and let $v$ be its generator. Recall that $\abs{-}_{s}$ denotes the $s$-adic norm. Then, there is some $k\geq 0$ such that every coordinate $v_i$ satisfies $\abs{s^kv_i}_s \leq 1$, but where $\abs{s^kv_i}_s = 1$ for at least one coordinate. Equivalently, $s^kv$ is a unimodular vector in $\cO^n$. This generates a rank 1 free summand inside $L[\frac{1}{s}]\cap \cO^n = L$, and so is equal to $L$. Therefore, $L$ was free to begin with. This calculation shows that $\ker\left(\Pic(\cO)\to \Pic( \cO[\frac{1}{s}])\right)$ is trivial, as required.
	\end{proof}

	\begin{lem}\label{lem:Chebotarev}
		Let $K$ be a number field, and $\cO_K$ its ring of integers. Then, there are infinitely many principal prime ideals in $\cO_K$.
	\end{lem}
	This result is a generalization of Dirichlet's theorem on primes in arithmetic progressions. See Chapter 8 of \cite{Marcus} for a lucid exposition. Theorem 48 in this book provides the framework, which is applied to the ideal class group on p. 231. The existence of the Hilbert class field is critical.
	
	We are now ready to prove Theorem \ref{Connectivity}, which asserts that $\cT_n(R)$ is Cohen-Macaulay when $R$ satisfies either $SR_2$ or is a Dedekind domain of arithmetic type.
	
	\begin{proof}[Proof of Theorem \ref{Connectivity}]
		First, we prove that, if $R$ satisfies $SR_2$ or is a Dedekind domain of arithmetic type, then $\cT_n(R)$ is $(n-3)$-connected for all $n$.
		
		By Theorems \ref{thm:SR2SK1}, \ref{surjStab}, and \ref{vdK}, any ring $R$ satisfying $SR_2$ additionally satisfies all of the hypotheses of Theorem \ref{TitsFiltConn}, so that $\cT_n(R)\cong \cT_{n,n-1}(R)$ is $(n-3)$-connected.
		
		Let $K$ be a number field, $S$ a finite set of places including all infinite places, and $\cO_S$ the associated Dedekind domain of arithmetic type. By Theorems \ref{E2}, \ref{surjStab}, and \ref{CFP}, if $\abs{S}\geq 2$ and $S$ contains a non-complex place, all of the hypotheses of Theorem \ref{TitsFiltConn} will be satisfied. So, it remains to check the theorem for the rings of integers $\cO_K$ of those number fields $K$ which do not embed into $\bbR$.
		
		In these cases, we apply Lemma \ref{lem:Chebotarev} to find a principal prime ideal $s\cO_K$. Since $S = S_\infty\cup \{\abs{-}_s\}$ contains the non-complex place $\abs{-}_s$, the above arguments show that $\cT_n(\cO_S)$ is $(n-3)$-connected. Since $\cO_S\cong \cO_K[\frac{1}{s}],$ it follows from Lemma \ref{CheekyLocalize} that $\cT_n(\cO_K)$ is also $(n-3)$-connected.
		
		We have shown that $\cT_n(R)$ is $(n-3)$-spherical for all $n$. It remains to show that it is Cohen-Macaulay. We use Proposition \ref{prop:QuillenCM}. Let $V$ be a free and cofree summand of $R^n$, and choose an isomorphism $V\cong R^k$. Using this isomorphism, we see that $\bbT_n(R)_{<V} \cong \bbT_k(R)$, and so is $(k-2)$-spherical. We additionally know that $\bbT_n(R)_{>V}$ is the poset of free and cofree summands of $(R^n/V)\cong R^{n-k}$. It follows that $\bbT_n(R)_{>V}\cong \bbT_{n-k}(R)$, and so it is $(n-k-2)$-spherical. Let $V'>V$ in $\bbT_n(R)$, and choose an isomorphism $\phi:V'\cong \bbR^{k'}$. Then, $(\bbT_n(R)_{>V}\cap \bbT_n(R)_{<V'})\cong \bbT_{k'}(R)_{<\phi(V)}\cong \bbT_{k'-k}(R),$ which is $(k'-k-2)$-spherical. Therefore, by Proposition \ref{prop:QuillenCM}, $\cT_n(R)$ is Cohen-Macaulay.
	\end{proof}
	
	\section{Generators for $\tilde{\HH}_{n-2}(\cT_n(R))$}
	
	In the remainder of this paper, we consider the structure of the Steinberg module:
	$$\St_n(R) := \tilde{\HH}_{n-2}(\cT_n(R);\bbZ).$$
	This is a free $\bbZ$-module. In this section, we are interested in finding generators of $\St_n(R)$, and (for finite $R$) computing its rank.
	
	First, we recall Definition \ref{defn:aptclass}. Let $\{v_1,\dots v_n\}$ be a free basis of $R^n$. Consider the poset $\cP([n])$ of proper nonempty subsets of $\{1,\dots, n\}$ under containment. The nerve of this poset (cf.\ Definition \ref{defn:nerve}) is the barycentrically subdivided boundary of the $(n-1)$-simplex $\sd\partial\Delta^{n-1}$, which is homeomorphic to $S^{n-2}$. We have a map of posets $$f\colon\cP([n])\to \bbT_n(R)$$ sending a subset $\{i_1, \dots, i_k\}$ to $\Span\{v_{i_1}, \dots, v_{i_k}\}$. The \emph{apartment class} $\left[\begin{array}{c|c|c}v_1 & \dots & v_n\end{array}\right]$ is the pushforward of the fundamental class $f_*[S^{n-2}]\in \tilde{\HH}_{n-2}(\cT_n(R);\bbZ)$.
	
	\subsection{Proof of Theorem \for{toc}{\ref*{mainthm:AptsGen}}\except{toc}{\ref{mainthm:AptsGen}}}\;
	
	Theorem \ref{mainthm:AptsGen} states that $\St_n(R)$ is generated by apartment classes when $R$ satisfies $SR_2$, or is a Dedekind domain of arithmetic type with a non-complex place. We prove the following, more general result:
	
	\begin{thm}\label{AptGen}
		If $R$ is a ring such that $\cB_n(R)$ and $\cT_n(R)$ are both Cohen-Macaulay, then $\St_n(R)$ is generated by apartment classes.
	\end{thm}
	\begin{proof}
		Notice that $C_{n-1}(\cB_n(R))$ is the abelian group of formal linear combinations of free bases of $R^n$. Let the map $$\Span: \cB'_n(R)\to \cT_n(R)$$ be the map defined at the end of \S3.3. Consider the composition: 
		$$\xymatrix{C_{n-1}(\cB_n(R))\ar[r]^\partial & \tilde{\HH}_{n-2}(\cB'_n(R))\ar[r]^(.6){\Span_*} & \St_n(R).}$$ where $\partial$ is the boundary map \[C_{n-1}(\cB_n(R)) \to C_{n-2}(\cB_n(R)) = C_{n-2}(\cB'_n(R)),\] whose image lies in $\tilde{\HH}_{n-2}(\cB'_n(R))\subseteq C_{n-2}(\cB'_n(R))$. 
		Since we have assumed that $\cB_n(R)$ is $(n-2)$-connected, $\tilde{\HH}_{n-2}(\cB_n(R)) $ is zero, and so the map $\partial$ must be surjective onto $\tilde{\HH}_{n-2}(\cB'_n(R))$. Lemma \ref{AptsSurj} implies the map $\Span_*$ is surjective. Therefore, the composite is surjective. It remains to show that $\Span_*\circ\;\partial\left(\{v_1,\dots, v_n\}\right)$ is equal to the apartment class $\left[\begin{array}{c|c|c}v_1 & \dots & v_n\end{array}\right]$. 
		
		By the definition of simplicial homology, a generator $\{v_1,\dots, v_n\}$ of $C_{n-1}(\cB_n(R))$ corresponds to a map $g\colon\Delta^{n-1} \to \cB_n(R)$, sending the $i$th vertex to $v_i$. We additionally know that $\partial \{v_1,\dots, v_n\} = g_*[\partial \Delta^{n-1}]$ inside $\HH_{n-2}(\cB_n'(R))$. Notice that, after subdividing $\partial \Delta^{n-1}$, the map $(\Span\circ g|_{\partial\Delta})$ is exactly the map $f$ used in the definition of the apartment class $\left[\begin{array}{c|c|c}v_1 & \dots & v_n \end{array}\right]$. Therefore, \[\Span_* \circ\;\partial\{v_1,\dots, v_n\} = (\Span \circ g)_*[\partial \Delta^{n-1}] =f_*[\partial \Delta^{n-1}] = \left[\begin{array}{c|c|c}v_1 & \dots & v_n \end{array}\right].\qedhere\]
	\end{proof}
	Theorem \ref{mainthm:AptsGen} is simply Theorem \ref{AptGen}, combined with Theorem \ref{Connectivity} (to verify that $\cT_n(R)$ is Cohen-Macaulay) and Theorems \ref{vdK} and \ref{CFP} (to verify that $\cB_n(R)$ is Cohen-Macaulay).
	
	\subsection{Proof of Theorem \for{toc}{\ref*{mainthm:UTApts}}\except{toc}{\ref{mainthm:UTApts}}}\;
	
	Given an apartment class $\left[\begin{array}{c|c|c}v_1 & \dots & v_n\end{array}\right]$, we can form a matrix whose $i$th column is the vector $v_i$. Since the $v_i$ form a basis of $R^n$, this matrix is in $\GL_n(R)$. We say that an apartment class is \emph{upper-triangular} if its associated matrix is strictly upper triangular (i.e., has 1 in all diagonal entries).
	When $R$ is a field, the Solomon-Tits theorem tells us that $\St_n(R)$ has a basis given by the upper-triangular apartment classes. In the case of a general commutative ring, Theorem \ref{mainthm:UTApts} states that these are still linearly independent, but will not span unless $R$ is a field.
	\begin{proof}[Proof of Theorem \ref{mainthm:UTApts}]
		First, we show linear independence. Associated to any top-dimensional simplex $\sigma$ of $\cT_n(R)$, there is a \emph{chamber map} $C_{n-2}(\cT_n(R)) \to \bbZ$, defined as the projection onto the summand corresponding to $[\sigma]$. Since $\HH_{n-2}(\cT_n(R))\subseteq C_{n-2}(\cT_n(R))$, these chamber maps restrict to $\St_n(R)$. Given an apartment class $A = \left[\begin{array}{c|c|c}v_1 & \dots & v_n\end{array}\right]$, let $c_A$ be the chamber map associated to the flag:
		$$0\subsetneq \Span\{v_n\}\subsetneq \Span\{v_{n-1},v_n\}\subsetneq\dots \subsetneq \Span\{v_2,\dots ,v_n\} \subsetneq R^n$$
		If $A$ is upper-triangular, we call such a flag a \emph{reverse upper-triangular flag}. We will show that if $A$ and $B$ are upper-triangular apartment classes, then $$c_A(B) = \begin{cases}
			1 & \text{if } A=B \\ 0 & \text{otherwise.}
		\end{cases}$$
		It is clear that $c_A(A) = 1$, from the definition of apartment classes. Assume now that $c_A(B) \ne 0$. We will show that $A=B$. Set $A = \left[\begin{array}{c|c|c}v_1 & \dots & v_n\end{array}\right]$ and $B = \left[\begin{array}{c|c|c}w_1 & \dots & w_n\end{array}\right]$, and let $\cF$ be the reverse upper-triangular flag associated to $A$. Since $c_A(B) \ne 0$, it follows that $B$ is supported on the simplex corresponding to $\cF$. Therefore, under some ordering, $\{w_i\}$ forms a flag-basis for $\cF$.
		
		First, we show that $w_n = v_n$. Since $\GL_1(R) = R^\times$, whichever element of $\{w_i\}$ spans the 1-dimensional subspace of $\cF$ must equal $uv_n$ for some $u\in R^\times$. The last entry of $v_n$ is nonzero. Since the ordered basis $(w_1,\dots, w_n)$ is upper-triangular, the only such element of $\{w_i\}$ is $w_n$. So, $w_n = uv_n$. Indeed, $u=1$, since the last entry of both $v_n$ and $w_n$ is 1.

		Likewise, the rank 2 summand must be $\Span\{w_n,w_{n-1}\}$, since $v_{n-1}$ has last coordinate 0 and second-last coordinate 1, and every other $w_j$ have the last two coordinates equal to zero. Continuing in this way, we see that $w_i = v_i$, so that $A=B$.
		
		It remains to prove that $\St_n(R)$ is not generated by upper-triangular apartment classes if $R$ is not a field. Pick a maximal ideal $\fm$ of $R$, and choose a nonzero element $m\in \fm$. Define $$\eta = \left[\begin{array}{c|c|c}  & 1 & \\ 1 & m & \\ & & \id_{n-2}\end{array}\right] + [\id_n]\in \St_n(R).$$
		Here we abuse notation by conflating an apartment class with the matrix produced by omitting some vertical bars. Notice that $c_A(\eta) = 0$, for any upper-triangular apartment class $A$. Indeed, the only upper-triangular flag on which either term of $\eta$ is supported is that associated to $[\id_n]$, but the relevant coefficients have opposite signs. Nevertheless, $\eta\ne 0$, since it has nonzero support on the flag
		\begin{align*}
			0\subsetneq\Span\{e_1+me_2\}\subsetneq\Span\{e_1, e_2\}\subsetneq\Span\{e_1, e_2&, e_3\}\subsetneq\dots\\ &\dots\subsetneq \Span\{e_1, \dots, e_{n-1}\}\subsetneq R^n.\qedhere
		\end{align*}
	\end{proof}
	\subsection{Proof of Theorem \for{toc}{\ref*{recdim}}\except{toc}{\ref{recdim}}}\;
	
	For the remainder of the paper, we specialize to the case of finite rings. In this case, $\St_n(R)$ is a finitely generated free $\bbZ$-module, and we are interested in its rank. 
	
	Recall that the Grassmannian $\Gr_k^n(R)$, $n\geq k\geq 0$, is the set of rank $k$ free and cofree direct summands of $R^n$. This is also the set of good flags of type $(k,n-k)$, and hence is a transitive $\GL_n(R)$-set. We note that $\abs{\Gr_0^n(R)} = 1$.
	
	Set $d_n = \rk \St_n(R)$ for $n\geq 1$, and define $d_0 = 1$. To prove Theorem \ref{recdim}, we must demonstrate the following recursive formula for $d_n$:
	$$d_0 = 1; \qquad \qquad d_n =\sum_{i=1}^n (-1)^{i-1} \abs{\Gr_{n-i}^n(R)} d_{n-i}\ \text{ for }n\geq 1.$$
	
	\begin{proof}[Proof of Theorem \ref{recdim}]
		Proposition \ref{FiltQuot} and Lemma \ref{AptsSurj} provide the following short exact sequence:
		$$0\to \tilde{\HH}_m(\cT_{n,m+1}(R))\to \bigoplus\tilde{\HH}_{m-1}(\cT_{m+1}(R))\to \tilde{\HH}_{m-1}(\cT_{n,m}(R)) \to 0.$$
		Setting $d_{n,m} = \rk \tilde{\HH}_{m-1}(\cT_{n,m}(R))$, this yields the recurrence $$d_{n,m+1} = \abs{\Gr_{m+1}^n(R)}d_{m+1} - d_{n,m}.$$ So, $$d_n = d_{n,n-1} = \abs{\Gr_{n-1}^n(R)}d_{n-1} - \abs{\Gr_{n-2}^n(R)}d_{n-2}+\dots \pm \abs{\Gr_{2}^n(R)}d_{2}\mp d_{n,1}.$$
		Since $\cT_{n,1}(R)$ is $\bbP^n(R)$ thought of as a discrete set, $d_{n,1} = \abs{\bbP^n(R)}-1$. Since $d_1 = \rk \tilde{\HH}_{-1}(\emptyset) = 1$, we can rewrite $d_{n,1}$ as $\abs{\Gr_1^n(R)} d_1 - \abs{\Gr_0^n(R)} d_0$. This produces the required formula.
	\end{proof}
	We remark that this formula reproduces the known rank formula for finite fields. Indeed, when $R=\bbF_q$, we will use Theorem \ref{recdim} to show that $$d_n = q^{\binom{n}{2}}.$$ Consider the following classical result:
		$$\prod_{k=0}^{n} (X+q^kY) = \sum_{k=0}^n q^{\binom{k}{2}} \abs{\Gr_k^n(\bbF_q)} X^{n-k}Y^k$$
	Setting $X = -1, Y=1$, we extract 
	$$0 =  \sum_{k=0}^n q^{\binom{k}{2}} \abs{\Gr_k^n(\bbF_q)} (-1)^{n-k}$$ 
	Rearranging, we have:
	$$q^{\binom{n}{2}} = \sum_{i=1}^{n} (-1)^{i+1} \abs{\Gr_{n-i}^n(\bbF_q)} q^{\binom{n-i}{2}}.$$
	Therefore $d_n = q^{\binom{n}{2}}$ satisfies the recurrence, as expected.
	
	\subsection{Some Computations with Theorem \for{toc}{\ref*{recdim}}\except{toc}{\ref{recdim}}}\;
	
	Outside of the field case, it seems difficult to produce a closed form for $d_n$. However, it is fairly accessible to both computer implementation and asymptotic analysis. As a first step, we find the cardinality of the Grassmannian. 
	
	We recall that a finite ring $R$ is necessarily semilocal, and so if $\cJ$ is its Jacobson radical, $R/\cJ$ is a product of fields.
	\begin{prop}\label{GrassSize}
		Let $R$ be a finite ring, and $\cJ$ be its Jacobson radical. Let $R/\cJ\cong F_1\times\dots \times F_\ell$. Then, 
		$$\abs{\Gr_k^n(R)} = \abs{\cJ}^{k(n-k)}\prod_{i=1}^\ell \abs{\Gr_k^n(F_i)}.$$
	\end{prop}
	\begin{proof}
		Since $\Gr_k^n(R)$ is a transitive $\GL_n(R)$-set, we can compute its cardinality using an orbit-stabilizer argument. 
		
		Recall that for an ideal $I$, we write $\Gamma_n(I) = \ker(\GL_n(R)\to\GL_n(R/I))$. We have $$\Gamma_n(I) = \left(\id + \Mat_n(I)\right)\cap \GL_n(R)$$ where $\Mat_n(I)$ is the ideal in $\Mat_n(R)$ of matrices all of whose entries lie in $I$. We observe that, since $\cJ$ is the Jacobson radical, $\left(\id + \Mat_n(I)\right)\subset\GL_n(R)$. Indeed, the determinant of any matrix in $(\id + \Mat_n(\cJ))$ lies in the coset $1 + \cJ$, which consists entirely of units. Therefore, $$\Gamma_n(\cJ) = \id + \Mat_n(\cJ).$$
		
		This observation tells us that 
		$$\abs{\GL_n(R)} = \abs{\cJ}^{n^2} \prod_{i=1}^\ell \abs{\GL_n(F_i)}.$$
		
		The stabilizer of $Re_1\oplus\dots\oplus Re_k\in \Gr_k^n(R)$ consists of matrices of the form:$$\begin{bmatrix} A & M \\ 0 & B \end{bmatrix} \text{ such that } A\in\GL_k(R),\ B\in\GL_{n-k}(R),\ M\in \Mat_{k,n-k}(R).$$ 
		
		So, we compute:
		\begin{align*}
			\abs{\Gr_k^n(R)} &= \dfrac{\abs{\GL_n(R)}}{\abs{\GL_k(R)}\abs{\GL_{n-k}(R)}\abs{\Mat_{k,n-k}(R)}}\\
			&= \dfrac{\abs{\cJ}^{n^2} \prod_{i=1}^\ell\abs{\GL_n(F_i)}}{\left(\abs{\cJ}^{k^2} \prod_{i=1}^\ell \abs{\GL_k(F_i)}\right)\left(\abs{\cJ}^{(n-k)^2} \prod_{i=1}^\ell \abs{\GL_{n-k}(F_i)}\right)\abs{R}^{k(n-k)}}\\
			&= \dfrac{\abs{\cJ}^{n^2}}{\abs{\cJ}^{k^2}\abs{\cJ}^{(n-k)^2} \abs{\cJ}^{k(n-k)}} \prod_{i=1}^\ell \dfrac{\abs{\GL_n(F_i)}}{\abs{\GL_k(F_i)}\abs{\GL_{n-k}(F_i)}\abs{F_i}^{k(n-k)}}\\
			&= \abs{\cJ}^{k(n-k)}\prod_{i=1}^\ell \abs{\Gr_k^n(F_i)}.\qedhere
		\end{align*}
	\end{proof}
	One upshot of Theorem \ref{recdim} and Proposition \ref{GrassSize} is that the dimension of $\St_n(R)$ does not depend too much on the algebraic structure of $R$, only on the cardinality of its Jacobson radical and of its residue fields. For example:
	\begin{exmp}
		The rings $\bbZ/p^2\bbZ$ and $\bbF_p[\epsilon]$ are not isomorphic. Indeed, their underlying abelian groups are not isomorphic. However, by Theorem \ref{recdim}, 
		$$\rk \St_n(\bbZ/p^2\bbZ)= \rk\St_n(\bbF_p[\epsilon]).$$
		This implies that $\cT_n(\bbZ/p^2\bbZ)$ and $\cT_n(\bbF_p[\epsilon])$ are homotopy equivalent.
	\end{exmp}
		
	Along these lines, we prove an asymptotic result. This shows that, in a sense, the upper-triangular classes span ``most" of the Steinberg representation for finite local rings.
	\begin{prop}
		Let $d_n(q,m) = \rk \St_n(R)$, where $(R,\fm)$ is any finite local ring such that $R/\fm = \bbF_q$ and $\abs{\fm} = m$. Fixing $n$, we have: 
		\begin{itemize}
			\item Asymptotically in $m$, $$d_n(q,m)\sim m^{\binom{n}{2}} \prod_{k=2}^n \abs{\Gr_{k-1}^k(\bbF_q)}.$$
			\item Asymptotically in $q$, $$d_n(q,m)\sim m^{\binom{n}{2}} q^{\binom{n}{2}} = \abs{R}^{\binom{n}{2}}.$$
		\end{itemize}
		These formulas are asymptotically equivalent with respect to $q$.
	\end{prop}
	\begin{proof}
		We prove this by induction on $n$. The base case $n=1$ is automatic, since $d_1(q,m) = 1$. Assume the first formula holds for all $n'<n$. Then, by Theorem \ref{recdim}, we have:
		\begin{align*}
			d_n(q,m)  &= \sum_{i=1}^n (-1)^{i-1} \abs{\Gr_{n-i}^n(R)} d_{n-i}(q,m)\\
			&= \sum_{i=1}^n (-1)^{i-1} \abs{\Gr_{n-i}^n(\bbF_q)}m^{i(n-i)} d_{n-i}(q,m)\\
			&\sim \sum_{i=1}^n\left( (-1)^{i-1} \abs{\Gr_{n-i}^n(\bbF_q)}m^{i(n-i)} m^{\binom{n-i}{2}} \prod_{k=2}^{n-i} \abs{\Gr_{k-1}^k(\bbF_q)}\right) \\
			&= \sum_{i=1}^n\left( (-1)^{i-1} \abs{\Gr_{n-i}^n(\bbF_q)}m^{\frac{1}{2}((n^2-n) - (i^2-i))} \prod_{k=2}^{n-i} \abs{\Gr_{k-1}^k(\bbF_q)}\right). 
		\end{align*}
		This is a polynomial in $m$, so is asymptotically equivalent to its leading term. By inspection, this is the $i=1$ term:
		$$d_n(q,m) \sim \abs{\Gr_{n-1}^n(\bbF_q)}m^{\binom{n}{2}} \prod_{k=2}^{n-1} \abs{\Gr_{k-1}^k(\bbF_q)} =  m^{\binom{n}{2}} \prod_{k=2}^n \abs{\Gr_{k-1}^k(\bbF_q)}.$$
		This proves the first formula. The second formula is similar, but we additionally need an approximation for $\abs{\Gr_k^n(\bbF_q)}$. The following is classical:
		$$\abs{\Gr_k^n(\bbF_q)} \sim q^{k(n-k)}$$ with respect to $q$. Using Theorem \ref{recdim} as above, we have:
		\begin{align*}
			d_n(q,m)  &= \sum_{i=1}^n (-1)^{i-1} \abs{\Gr_{n-i}^n(R)} d_{n-i}(q,m)\\
			&= \sum_{i=1}^n (-1)^{i-1} \abs{\Gr_{n-i}^n(\bbF_q)}m^{i(n-i)} d_{n-i}(q,m)\\
			&\sim \sum_{i=1}^n\left( (-1)^{i-1} q^{i(n-i)}m^{i(n-i)} q^{\binom{n-i}{2}}m^{\binom{n-i}{2}} \right) \\
			&= \sum_{i=1}^n (-1)^{i-1} q^{\frac{1}{2}((n^2-n) - (i^2-i))}m^{\frac{1}{2}((n^2-n) - (i^2-i))}.  
		\end{align*}
		As above, this is a polynomial in $q$, thus asymptotically equivalent to the $i=1$ term:
		\[d_n(q,m)\sim m^{\binom{n}{2}} q^{\binom{n}{2}}.\qedhere\]
	\end{proof}
		
	\subsection{Proof of Theorem \for{toc}{\ref*{mainthm:cohomology}}\except{toc}{\ref{mainthm:cohomology}}}\;
	
	In this section, we produce a lower bound on the top-dimensional cohomology of certain principal congruence subgroups. The relevant tool is the following theorem of Borel and Serre.
	\begin{thm}[Borel-Serre, \cite{BS}]
		Let $n\geq 2$. Let $K$ be a number field, and $\cO$ its ring of integers. Let $\nu$ be the rational cohomological dimension of $\SL_n(\cO)$.
		Then, for any finite-index subgroup $\Gamma$ of $\SL_n(\cO)$ and any $\bbQ[M]$-module $M$, there is an isomorphism $$\HH^q(\Gamma; M) \cong \HH_{\nu-q}(\Gamma; \St_n(K)\otimes_\bbQ M).$$
	\end{thm}
	In particular, 
		$$\HH^{\nu}(\Gamma;\bbQ) \cong \HH_0(\Gamma; \St_n(K)) \cong (\St_n(K))_\Gamma$$
	where $(\St_n(K))_\Gamma$ is the module of coinvariants.
	\begin{proof}[Proof of Theorem \ref{mainthm:cohomology}]
		Let $\cO$ be a number ring. We aim to prove that \[\dim \HH^{\nu}(\Gamma_n(I);\bbQ)\geq \rk \St_n(\cO/I).\] By Borel-Serre's theorem, it suffices to find a surjective $\SL_n(\cO)$-equivariant map from $\St_n(K)$ to $\St_n(\cO/I)$. Since $\St_n(\cO/I)$ is $\Gamma_n(I)$-invariant, any such map will factor through a surjective map $$(\St_n(K))_{\Gamma_n(I)}\to \St_n(\cO/I).$$ This implies the desired inequality.
		
		We construct the desired map. Let $S'$ be the set of places corresponding to prime ideals $\fp$ such that $I\not\subseteq \fp$, and let $S= S'\cup S_\infty$. Notice that $\cO_S$ is semilocal, and so $\Pic(\cO_S) = 1$. Therefore, there is a $\SL_n(\cO)$-equivariant homeomorphism $\cT_n(\cO_S)\cong \cT_n(K)$. Notice also that $I\cO_S$ is a proper ideal of $\cO_S$, and that $\cO_S/I\cO_S \cong \cO/I$. This induces an $\SL_n(\cO)$-equivariant map $\cT_n(\cO_S)\to \cT_n(\cO/I)$. Taking homology, we have an equivariant map
		$$\St_n(K)\cong \St_n(\cO_S)\to \St_n(\cO/I).$$ It suffices to show that the rightmost map is surjective. Call this map $\phi$. Since $I$ is a nonzero ideal, $\cO/I$ is finite, and so satisfies $SR_2$. Therefore, Theorem \ref{mainthm:AptsGen} applies, so that $\St_n(\cO/I)$ is generated by apartment classes. 
		
		We show that each apartment class in $\St_n(\cO/I)$ is the image of some apartment class in $\St_n(\cO)$, so that $\phi$ is surjective. Let $A$ be an apartment class in $\St_n(\cO/I)$, and say that $A=\left[\begin{array}{c|c|c}v_1 & \dots & v_n \end{array}\right]$. These vectors are a basis of $(\cO/I)^n$, but this basis is not necessarily the image of a basis of $\cO^n$. However, it may be transformed into such a basis without changing $A$. There is a unique $u\in (\cO/I)^\times$ such that the matrix with columns $uv_1,v_2,\dots, v_n$ lies in $\SL_n(\cO/I)$. Call this matrix $M$. Since $\cO/I$ satisfies $SR_2$, the group $\SL_n(\cO/I)$ is generated by elementary matrices, and hence the map $\SL_n(\cO)\to \SL_n(\cO/I)$ is surjective. Let $\widetilde{M}$ be a lift of $M$ in $\SL_n(\cO)$, and let the columns of $\widetilde{M}$ be $w_1,\dots,w_n$. Then, by construction \[\phi(\left[\begin{array}{c|c|c}w_1 & \dots & w_n \end{array}\right]) = \left[\begin{array}{c|c|c}uv_1 & \dots & v_n \end{array}\right].\] Since rescaling a vector does not change its span, the right hand side is equal to $A$. Since $A$ was arbitrary, we have shown that the image of $\phi$ includes all apartment classes. Therefore $\phi$ is surjective, as required.
	\end{proof}
	\section{$\tilde{\HH}_{n-2}(\cT_n(R))$ as a $\GL_n(R)$-representation.}
	We use the following notation: If $K$ is a field, then $$\St_n^K(R):= \HH_{n-2}(\cT_n(R);K) \cong \St_n(R)\tensor K.$$
	The natural action of $\GL_n(R)$ on $\cT_n(R)$ makes $\St_n^K(R)$ a $\GL_n(R)$-representation. 
	
	Recall that Steinberg \cite{Steinberg} and Putman-Snowden \cite{SteinIrred} prove that $\St_n^K(F)$ is irreducible when $F$ is any field and $K$ is of characteristic zero. We prove a converse:
	\begin{prop}\label{SteinReducible}
		If $R$ is not a field, then $\St_n^K(R)$ is not irreducible.
	\end{prop}
	\begin{proof}
		Let $\fm$ be a maximal ideal of $R$, and define the subspace $$V= \ker(\phi:\St_n^K(R)\to \St_n^K(R/\fm)).$$ Since $\phi$ is a map of $\GL_n(R)$-representations, $V$ is a subrepresentation of $\St_n^K(R)$. Notice that, if $A$ is an upper-triangular apartment class in $\St_n^K(R)$, then so is $\phi(A)$ in $\St_n^K(R/\fm)$. Therefore, $\phi$ is not zero, and so $V$ is proper. The class $\eta$ defined in the proof of Theorem \ref{mainthm:UTApts} lies in $V$, so $V$ is nonzero. Therefore, $\St_n^K(R)$ is reducible.
	\end{proof}
	
	For the remainder of the section, let $R$ be a finite ring.  Let $I$ be an ideal of $R$. There is a natural map $\cT_n(R)\to\cT_n(R/I)$, which induces a map $\St_n(R)\to \St_n(R/I)$ on homology. Since this is a map of representations and $\St_n(R/I)$ is $\Gamma_n(I)$-invariant, we get a map $\St_n^K(R)^{\Gamma_n(I)}\to \St_n^K(R/I)$. Theorem \ref{mainthm:coinvariants} states that this is an isomorphism when $I$ lies inside the Jacobson radical and $K$ is of characteristic zero:
	
	\begin{proof}[Proof of Theorem \ref{mainthm:coinvariants}]
		Let $\tilde{C}_\bullet(X,K)$ denote the augmented chain complex associated to the simplicial complex $X$. First, we show that $\tilde{C}_q(\cT_n(R);K)^{\Gamma_n(I)} \cong \tilde{C}_q(\cT_n(R/I);K)$. Since $\tilde{C}_q(\cT_n(R);K)$ is the module of formal linear combinations of flags of length $q+1$, it follows that $\tilde{C}_q(\cT_n(R);K)^{\Gamma_n(I)}$ is the submodule such that the coefficients of flags are constant along $\Gamma_n(I)$-orbits. Using Lemma \ref{flagTrans}, we can understand $\tilde{C}_q(\cT_n(R);K)^{\Gamma_n(I)}$ as a permutation representation. Let $\lambda$ be a partition of $n$, and let $P_{\lambda}(R)$ be the stabilizer of the standard flag $\Span_\lambda (e_1,\dots, e_n)$. Then, $\tilde{C}_q(\cT_n(R);K)$ is the permutation representation on the set:
		$$\bigsqcup_{\abs{\vec{\lambda}} = q+1} \Fl_\lambda(R) \cong \bigsqcup_{\abs{\vec{\lambda}} = q+1} \GL_n(R)/P_{\vec{\lambda}}(R).$$
		The $\Gamma_n(I)$-orbits of this are in bijection with the double cosets
		$$\bigsqcup_{\abs{\vec{\lambda}} = q+1} \Gamma_n(I)\backslash\GL_n(R)/P_{\vec{\lambda}}(R).$$
		Notice that $\Gamma_n(I)$ is normal, so this is still naturally a left $\GL_n(R)$-set. Since $I$ is contained in the Jacobson radical, $\GL_n(R)$ surjects onto $\GL_n(R/I)$, so we have:
		$$\bigsqcup_{\abs{\vec{\lambda}} = q+1} \Gamma_n(I)\backslash\GL_n(R)/P_{\vec{\lambda}}(R)\cong \bigsqcup_{\abs{\vec{\lambda}} = q+1} \GL_n(R/I)/\text{im}(P_{\vec{\lambda}}(R)).$$
		Here, $\text{im}(P_{\vec{\lambda}}(R))$ denotes the image of $P_{\vec{\lambda}}(R)$ in $\GL_n(R/I)$. Notice that $P_\lambda(R)$ consists of block matrices, where a diagonal block lives in $\GL_{\lambda_i}(R)$, a block below the diagonal is zero, and a block above the diagonal is an arbitrary matrix. Since $I$ lies in the Jacobson radical, $\text{im}(P_{\vec{\lambda}}(R)) = P_\lambda(R/I)$, and so:
		 
		$$\bigsqcup_{\abs{\vec{\lambda}} = q+1} \GL_n(R/I)/\text{im}(P_{\vec{\lambda}}(R)) = \bigsqcup_{\abs{\vec{\lambda}} = q+1} \GL_n(R/I)\left/P_{\vec{\lambda}}(R/I)\right. \cong \bigsqcup_{\abs{\vec{\lambda}} = q+1} \Fl_\lambda(R/I).$$
		Therefore, $\tilde{C}_q(\cT_n(R);K)^{\Gamma_n(I)}$ and $\tilde{C}_q(\cT_n(R/I);K)$ are permutation representations on the same $\GL_n(R)$-set, and so are isomorphic.
		
		It follows easily that this extends to an isomorphism of complexes, since the boundary maps of both complexes are given by alternating sums of maps of these $\GL_n(R)$-sets which forget certain subspaces in the flag. Clearly these commute with our chosen isomorphisms, and so $\tilde{C}_\bullet(\cT_n(R);K)^{\Gamma_n(I)} \cong \tilde{C}_\bullet(\cT_n(R/I);K)$ as complexes.
		
		Since $\Gamma_n(I)$ is finite and $K$ has characteristic zero, the functor $(-)^{\Gamma_n(I)}$ is exact. Therefore this isomorphism of complexes induces an isomorphism 
		\begin{align*}
			\St_n^K(R)^{\Gamma(I)} = \tilde{\HH}_{n-2}(\cT_n(R);K)^{\Gamma(I)} &\cong \HH_{n-2}(\tilde{C}_\bullet(\cT_n(R);K)^{\Gamma_n(I)})\\ &\cong \HH_{n-2}(\tilde{C}_\bullet(\cT_n(R/I);K)) = \St_n^K(R/I)\qedhere
		\end{align*}
	\end{proof}
	
	If $(R,\fm)$ is a finite local ring, then $\fm$ is its Jacobson radical, and so Theorem \ref{mainthm:coinvariants} implies that the length of $\St^\bbQ_n(R)$ as a $\GL_n(R)$-representation is at least the length of $R$. We want to know, in general, what the length of $\St^K_n(R)$ is as a $\GL_n(R)$-representation.
	
	This seems quite difficult to compute in general.  We will solve the $n=2$ case for $\bbZ/p^k\bbZ$.
	
	First, we recall a strengthening of locality from commutative algebra:
	\begin{defn}
		A ring is \emph{uniserial} if its ideals are linearly ordered by inclusion.
	\end{defn}
	It is easily seen that $\bbZ/p^k\bbZ$ is uniserial. In fact:
	\begin{prop}\label{prop:DedekindUniserial}
		Let $R$ be a Dedekind domain, and $\fp$ a maximal ideal. Then $R/\fp^k$ is uniserial.
	\end{prop}
	\begin{proof}
		The localization $R_\fp$ is a DVR, since $R$ is a Dedekind domain. Every DVR is uniserial, and therefore so are its quotients. Since $\fp$ is maximal, $R/\fp^k$ is local at $\fp$, and it follows that $R/\fp^k \cong R_\fp/\fp^k$, whence the proposition.
	\end{proof}
	\begin{exmp}
		The ring $\bbF_2[x,y]/(x^2,xy,y^2)$ is a finite local ring that is \emph{not} uniserial. Indeed, its lattice of ideals is: $$
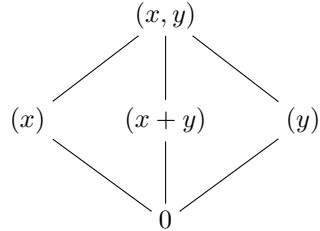
\xymatrix{ & (x,y) \ar@{-}[dr] \ar@{-}[d] \ar@{-}[dl]& \\ (x) \ar@{-}[dr] & (x+y) \ar@{-}[d] & (y) \ar@{-}[dl]\\ & 0 & }$$
	\end{exmp}
	We also need the following result from Mackey theory:
	\begin{lem}
		Let $G$ be a finite group, and $X$ a transitive $G$-set. Let $K$ be a field of characteristic zero. Denote by $K[X]$ the permutation representation associated to $X$. Then, as vector spaces,
		$$\End_G(K[X])\cong K[(X\times X)/G]$$
	\end{lem}
	\begin{proof}
		Since $X$ is a transitive $G$-set, we may write $X\cong G/H$ for some subgroup $H$ of $G$. By Mackey's theorem (\cite[Theorem 32.1]{Bump}, in the case $H_1=H_2$ and $V_1=V_2 = K_{\text{triv}}$), $\End_G(K[X])$ is isomorphic to $K[H\backslash G /H]$. We demonstrate a bijection $$H\backslash G /H\longleftrightarrow (X\times X)/G$$
		Let $(x_1,x_2)\in X\times X$. Since $X$ is transitive, this is in the $G$-orbit of $(H,gH)$, for some $g\in G$. It is also in the orbit of $(H,g'H)$ if and only if $g'H = h\cdot(gH)$ for some $h\in H$, since $H$ is the stabilizer of the coset $H$. So, associated to the orbit of $(x_1,x_2)$ is a unique double coset $HgH$. This produces a map $(X\times X)/G\to H\backslash G/H$. To produce an inverse, one associates to the double coset $HgH$ the orbit of $(H,gH)$. 
		
		We have exhibited a bijection between $H\backslash G /H$ and $(X\times X)/G$, and hence proved that $K[H\backslash G /H]\cong K[(X\times X)/G]$ as vector spaces.
	\end{proof}
	The following result is a strengthening of Theorem \ref{mainthm:length2} (cf.\ Proposition \ref{prop:DedekindUniserial}).
	\begin{thm}
		Let $(R,\fm)$ be a finite uniserial ring, and $k$ the length of $R$. Let $K$ be a field of characteristic zero. Then, $\St^K_2(R)$ has exactly $k$ irreducible summands. Each irreducible summand appears exactly once, and none are trivial.
	\end{thm}
	\begin{proof}
		Theorem \ref{mainthm:coinvariants} implies that $\St_2^K(R/\fm^{k-1})$ is a direct summand of $\St_n^K(R)$. By induction on $k$, $\St_2^K(R/\fm^{k-1})$ has exactly $k-1$ distinct irreducible summands. Since $\St_2^K(R)$ is not $\Gamma_2(\fm)$-invariant, $\St_2^K(R/\fm^{k-1})$ is a proper subspace, and so $\St_2^K(R)$ has at least $k$ distinct irreducible summands.
		
		We now prove an upper bound. Notice that $$C_0(\cT_2(R),K)\cong \St_2^K(R)\oplus K,$$ and that $C_0(\cT_2(R))$ is the permutation representation $K[\bbP^1(R)]$. By the previous lemma, we know that $$\End_{\GL_2}\left(C_0(\cT_2(R))\right)\cong K[(\bbP^1\times \bbP^1)/\GL_2].$$ The dimension of the right hand side is the number of orbits of $\bbP^1\times \bbP^1$.
		
		Consider the pair of lines $(L_1,L_2)$, which are generated by $v_1,v_2$ respectively. Let $L_1\cap L_2$ be the ideal $I\leq R$. If $I=0$, then the map $$R^2\to R^2:= (x,y)\mapsto (x\;v_1+y\;v_2)$$
		is an element of $\GL_2$ taking $(Re_1,Re_2)$ to $(L_1,L_2)$.
				
		Assume now $I\ne 0$. Acting by $\GL_2$, we may take $L_1 = Re_1$, $v_1=e_1$. Write $v_2 = ae_1+be_2$. We know $b\in \text{Ann}(I)$, which is a proper ideal. Since $v_2$ is unimodular and $R$ is local, we know $a$ must be a unit. Without loss of generality, $a=1$. So $L_2 = R(e_1 + be_2)$, where $b\in \text{Ann}(I)$, but is in no proper submodule of $\text{Ann}(I)$. By uniseriality, $\text{Ann}(I)$ is cyclic as an $R$-module with generator $i_0$, and there is some $u\in R^\times$ such that $b=ui_0$. So, we may act by $\GL_2$ to take $(L_1,L_2)$ to $(Re_1,R(e_1+i_0\;e_2))$.
		
		The above argument shows that $\bbP^1\times \bbP^1$ consists of $k+1$ orbits as a $\GL_2$-set. This shows that $\End_{\GL_2}\left(C_0(\cT_2(R))\right)$ is $(k+1)$-dimensional, and so its length as a $\GL_2$-representation is at most $k+1$. But we already know its length is at least $k+1$. Therefore $C_0(\cT_2(R))$ has exactly $k+1$ irreducible summands, and so $\St_2^K(R)$ has exactly $k$ irreducible summands.
		
		It is classical that when the length of a representation is equal to the dimension of its endomorphism ring, every irreducible summand appears exactly once. This implies that all summands of $\St_2^K(R)$ are all distinct, and that none are trivial.
	\end{proof}
	
	\bibliographystyle{plain}
	\bibliography{bibliography}

\end{document}